\documentclass[12pt, a4paper]{amsart}
\usepackage{amssymb,amsmath,amsthm, amscd, enumerate, mathrsfs}
\usepackage{graphicx, hhline}
\usepackage[all]{xy}
\usepackage[usenames]{color}
\usepackage[]{hyperref}
\usepackage{hyperref}

\hypersetup{colorlinks=true}

\theoremstyle{plain}
\newtheorem{thm}{Theorem}[section] 
\newtheorem{cor}[thm]{Corollary}

\newtheorem{conj}[thm]{Conjecture}
\newtheorem{propdef}[thm]{Proposition-Definition} 

\newtheorem{lem}[thm]{Lemma}
\newtheorem{mainthm}{Theorem}
\newtheorem{maincor}{Corollary}
\theoremstyle{definition} 
\newtheorem{defn}[thm]{Definition}
\newtheorem{eg}[thm]{Example} 
\theoremstyle{remark}
\newtheorem{rem}[thm]{Remark}

\newtheorem*{cl}{Claim}

\newtheorem*{acknowledgement}{Acknowledgments}

\newcommand{\sO}{\mathcal{O}}
\newcommand{\J}{\mathcal{J}} 
\newcommand{\Z}{\mathbb{Z}}
\newcommand{\N}{\mathbb{N}} 
\newcommand{\Q}{\mathbb{Q}} 
\newcommand{\R}{\mathbb{R}} 
\newcommand{\C}{\mathbb{C}}
\newcommand{\F}{\mathbb{F}}
\newcommand{\ba}{\mathfrak{a}}
\newcommand{\m}{\mathfrak{m}}

\newcommand{\adj}{\mathop{\mathrm{adj}}\nolimits}
\newcommand{\Hom}{\mathop{\mathrm{Hom}}\nolimits}
\newcommand{\Spec}{\mathop{\mathrm{Spec}}\nolimits}
\newcommand{\Proj}{\mathop{\mathrm{Proj}}\nolimits}
\newcommand{\codim}{\mathop{\mathrm{codim}}\nolimits}

\newcommand{\lct}{\mathop{\mathrm{lct}}\nolimits}
\newcommand{\fpt}{\mathop{\mathrm{fpt}}\nolimits}
\newcommand{\bigzerol}{\smash{\lower1.0ex\hbox{\bg 0}}}

\renewcommand{\labelenumi}{\rm{(\theenumi)}}
\newfont{\bg}{cmr17 scaled\magstep5}

\title{Adjoint ideals and a correspondence between\\ log canonicity and $F$-purity}
\author{Shunsuke Takagi}
\address{Graduate School of Mathematical Sciences, University of Tokyo, 3-8-1 Komaba, Meguro-ku, Tokyo 153-8914, JAPAN}
\email{stakagi@ms.u-tokyo.ac.jp}

\dedicatory{Dedicated to Professor~Shihoko Ishii on the~occasion of her~sixtieth~birthday.}

\begin{document}

\begin{abstract}
This paper presents three results on $F$-singularities.
First, we give a new proof of Eisenstein's restriction theorem for adjoint ideal sheaves, using the theory of $F$-singularities. 
Second,  we show that a conjecture of Musta\c{t}\u{a} and Srinivas \cite[Conjecture 1.1]{MS} implies a conjectural correspondence of $F$-purity and log canonicity.  
Finally, we prove this correspondence when the defining equations of the variety are very general.  
\end{abstract}

\maketitle
\markboth{SHUNSUKE TAKAGI}{A CORRESPONDENCE BETWEEN LOG CANONICITY AND $F$-PURITY}


\section*{Introduction}
This paper deals with the theory of $F$-singularities, which are singularities defined using the Frobenius morphism in positive characteristic. 
We present three main results. 
First, we give a new proof of a restriction theorem for adjoint ideal sheaves. 
Second, we show that a certain arithmetic conjecture implies a conjectural correspondence of $F$-purity and log canonicity.  
Finally, we prove this correspondence when the defining equations of the variety are very general.   

The notion of the adjoint ideal sheaf along a normal $\Q$-Gorenstein closed subvariety $X$ of a smooth complex variety $A$ with codimension $c$ was introduced in \cite{Ta2}  (see Definition \ref{adjoint ideal} for its definition). 
It is a modification of the multiplier ideal sheaf associated to the pair $(A, cX)$ and encodes much information on the singularities of $X$. 
Eisenstein \cite{Ei} recently proved a restriction theorem for these adjoint ideal sheaves. 
In this paper, we give a new proof of his result, using the theory of $F$-singularities. 

Building on earlier results \cite{HY}, \cite{Ta5} and \cite{Ta4}, the author 
introduced in \cite{Ta2} a positive characteristic analogue of the adjoint ideal sheaf, called the test ideal sheaf (see Definition \ref{test ideal def}). 
He conjectured that the adjoint ideal sheaf coincides, after reduction to characteristic $p \gg 0$, with the test ideal sheaf, and some partial results were obtained in \textit{loc.~cit}. 
Making use of these results, we reduce the problem to an ideal theoretic problem on a normal $\Q$-Gorenstein ring essentially of finite type over a perfect field of characteristic $p>0$. 
The desired restriction formula is then obtained by adapting the argument of \cite{Sc2} (which can be traced back to \cite{Fe}) to our setting (see Theorem \ref{restriction}). 
As a corollary, we show the correspondence between adjoint ideal sheaves and test ideal sheaves in a full generality:

\renewcommand{\themaincor}{3.4}
\begin{maincor}
Let $t > 0$ be a real number and $Z$ be a proper closed subscheme of $A$ which does not contain $X$ in the support.
We denote by $\adj_X(A,tZ)$ the adjoint ideal sheaf of the pair $(A,tZ)$ along $X$, and let $(A_p, X_p, Z_p,  \adj_X(A,tZ)_p)$be a reduction modulo $p \gg 0$ of $(A, X, Z, \adj_X(A,tZ))$. 
If $\widetilde{\tau}_{X_p}(A_p,tZ_p)$ is the test ideal sheaf of $(A_p, tZ_p)$ along $X_p$, then 
$$\widetilde{\tau}_{X_p}(A_p,tZ_p)=\adj_X(A,tZ)_p.$$
\end{maincor}

The other ingredients of this paper are on a correspondence between $F$-pure singularities and log canonical singularities. 
$F$-pure singularities are defined via splitting of Frobenius morphisms (see Definition \ref{F-sing}). 
Log canonical singularities form a class of singularities associated to the minimal model program  (see Definition \ref{sing of pairs}). 
It is known that the pair $(X; tZ)$ is log canonical if its modulo $p$ reduction $(X_p; t Z_p)$ is $F$-pure for infinitely many primes $p$, and the converse is conjectural (see Conjecture \ref{lc conj} for the precise statement). 
This conjecture is widely open and only a few special cases are known. 
On the other hand, Musta\c{t}\u{a}--Srinivas \cite[Conjecture 1.1]{MS} proposed the following more arithmetic conjecture to study a behavior of test ideal sheaves: if $V$ is a $d$-dimensional smooth projective variety over an algebraically closed field of characteristic zero, then the action induced by the Frobenius morphism on the cohomology group $H^d(V_p, \sO_{V_p})$ of its modulo $p$ reduction $V_p$ is bijective for infinitely many primes $p$. 
In this paper, we show that their conjecture implies the correspondence of $F$-purity and log canonicity (see Theorem \ref{implication}). 
Our result can be viewed as strong evidence in favor of this conjectural correspondence, although the conjecture of Musta\c{t}\u{a}--Srinivas is also largely open.

As additional evidence of this correspondence, we consider the case when the defining equations of $X$ are very general.  
Shibuta and the author \cite{ST} proved the correspondence if $X=\mathbb{C}^n$ and $Z$ is a complete intersection binomial subscheme or a space monomial curve. 
Using a similar idea, Hern\'andez \cite{He} recently proved the case 
when $X=\mathbb{C}^n$ and $Z$ is a hypersurface of $X$ such that the coefficients of terms of its defining equation are algebraically independent over $\Q$. 
Using the techniques we have developed for Theorem \ref{restriction}, 
we generalize his result as follows: 
\renewcommand{\themainthm}{4.1}
\begin{mainthm}
Let $X \subseteq \mathbb{C}^n=\Spec \C[x_1, \dots, x_n]$ be a normal $\Q$-Gorenstein closed subvariety of codimension $c$ passing through the origin $0$. 
Let $r$ denote the Gorenstein index of $X$ and $J_X$ denote the l.c.i.~defect ideal of $X$.
Let $\ba \subseteq \sO_X$ be a nonzero ideal and $t >0$ be a real number. 
Suppose that there exist a system of generators $h_1, \dots, h_l$ for the defining ideal $\mathcal{I}_X$ of $X$  and a system of generators $h_{l+1}, \dots, h_{\nu}$ for $\ba$ with the following property: for each $i=1, \dots, \nu$, we can write 
$$h_i=\sum_{j=1}^{\rho_i}\gamma_{ij}x_1^{\alpha^{(1)}_{ij}} \cdots x_n^{\alpha^{(n)}_{ij}} \in \C[x_1, \dots, x_n]  
\; \left((\alpha^{(1)}_{ij}, \dots, \alpha^{(n)}_{ij}) \in \Z_{\ge 0}^n \setminus \{\mathbf{0}\}, \; \gamma_{ij} \in \C^* \right),$$
where $\gamma_{i1}, \dots, \gamma_{i\rho_i}$ are algebraically independent over $\Q$. 
Then $(X; t V(\ba)+\frac{1}{r} V(J_X))$ is log canonical at $0$ if and only if its modulo $p$ reduction is $F$-pure at $0$ for infinitely many primes $p$.  
\end{mainthm}

\begin{small}
\begin{acknowledgement}
The author is grateful to Daniel Hern\'andez, Nobuo Hara, Junmyeong Jang, Masayuki Kawakita, Karl Schwede and Takafumi Shibuta for valuable conversations. 
He is indebted to Natsuo Saito for his help with LaTeX. 
He also would like to thank Shihoko Ishii and anonymous referee for pointing out mistakes in a previous version of this paper. 
The author would like to express his gratitude to the Massachusetts Institute of Technology, where a part of this work was done, for their hospitality during the winter of 2010--2011. The author was partially supported by Grant-in-Aid for Young Scientists (B) 20740019 and 23740024 from JSPS and by Program for Improvement of Research Environment for Young Researchers from SCF commissioned by MEXT of Japan. 
\end{acknowledgement}
\end{small}


\section{Preliminaries}
\subsection{Test ideals and $F$-singularities of pairs}
In this subsection, we briefly review the definitions of test ideal sheaves and $F$-singularities of pairs. 
The reader is referred to \cite{Sc}, \cite{Sc2}, \cite{Sc3}, \cite{Ta3} and \cite{Ta2} for the details.

Throughout this paper, all schemes are Noetherian, excellent and separated, and all sheaves are coherent. 
Let $A$ be an integral scheme of prime characteristic $p$. 
For each integer $e \ge 1$,  we denote by $F^e: A \to A$ or $F^e: \sO_A \to F^e_*\sO_A$ the $e$-th iteration of the absolute Frobenius morphism on $A$. 
We say that $A$ is \textit{$F$-finite} if $F: A \to A$ is a finite morphism. 
For example, every scheme essentially of finite type over a perfect field is $F$-finite. 
Given an ideal sheaf $I \subseteq \sO_A$, for each $q=p^e$, we denote by $I^{[q]} \subseteq \sO_A$ the ideal sheaf identified with $I \cdot F^e_*\sO_A$ via the identification $F^e_*\sO_A \cong \sO_A$. 
For a closed subscheme $Y$ of $A$, we denote by $\mathcal{I}_Y$ the defining ideal sheaf of $Y$ in $X$.

The notion of test ideal sheaves along arbitrary subvarieties was introduced in \cite{Ta2}. 
Below we give an alternate description of these sheaves based upon the ideas of \cite{Sc2}.
Let $A$ be a normal $\Q$-Gorenstein variety over an $F$-finite field of characteristic $p>0$ 
and $X \subseteq A$ be a reduced equidimensional closed subscheme of codimension $c$. 
Suppose that the Gorenstein index of $A$ is not divisible by $p$. 
There then exists infinitely many $e$ such that $(p^e-1)K_A$ is Cartier, and we fix such an integer $e_0 \ge 1$.  
Grothendieck duality yields an isomorphism of $F^{e_0}_*\sO_A$-modules
$$F^{e_0}_*\sO_A \cong \mathscr{H}\mathrm{om}_{\sO_A}(F^{e_0}_*((1-p^{e_0})K_A), \sO_A),$$
and we denote by 
$$\varphi_{A, e_0}:F^{e_0}_*\sO_A((1-p^{e_0})K_A) \to \sO_A$$
the map corresponding to the global section $1$ of $\sO_A$ via this isomorphism. 
When $A$ is Gorenstein, we can describe $\varphi_{A, e_0}$ more explicitly: it is obtained by tensoring the canonical dual $(F^{e_0})^{\vee}: F^{e_0}_*\omega_A \to \omega_A$ of the $e_0$-times iterated Frobenius morphism $F^{e_0}:\sO_A \to F^{e_0}_*\sO_A$ with $\sO_A(-K_A)$. 
Also, the composite map 
$$\varphi_{A, e_0} \circ F^{e_0}_*\varphi_{A, e_0} \circ \cdots \circ F^{(n-1)e_0}_*\varphi_{A, e_0}:F^{ne_0}_*\sO_A((1-p^{ne_0})K_A) \to \sO_A.$$
is denoted by $\varphi_{A, ne_0}$ for all integers $n \ge 1$. 
Just for convenience, $\varphi_{A, 0}$ is defined to be the identity map $\sO_A \to \sO_A$. 

\begin{propdef}[\textup{cf.~\cite[Definition 2.2]{Ta2}}]\label{test ideal def}
Let the notation be as above and let $Z=\sum_{i=1}^m t_i Z_i$ be a  formal combination where the $t_i$ are nonnegative real numbers and the $Z_i$ are proper closed subschemes of $A$ which do not contain any component of $X$ in their support.
\begin{enumerate}
\item 
There exists a unique smallest ideal sheaf $J \subseteq \sO_A$ whose support does not contain any component of $X$ and which satisfies  
$$\qquad \quad \varphi_{A,ne_0}(F^{ne_0}_*(J\mathcal{I}_X^{c(p^{ne_0}-1)}\mathcal{I}_{Z_1}^{\lceil t_1(p^{ne_0}-1) \rceil} \cdots \mathcal{I}_{Z_m}^{\lceil t_m(p^{ne_0}-1) \rceil} \sO_A((1-p^{ne_0})K_A))) \subseteq J$$ 
for all integers $n \ge 1$. 
This ideal sheaf is denoted by $\widetilde{\tau}_X(A,Z)$. 
When $X=\emptyset$ $($resp.~$Z=\emptyset)$, we denote this ideal sheaf simply by $\widetilde{\tau}(A; Z)$ $($resp.~$\widetilde{\tau}_X(A))$.
\item
$(A, Z)$ is said to be \textit{purely $F$-regular} along $X$ if $\widetilde{\tau}_X(A,Z)=\sO_{A}$. 
\end{enumerate}
\end{propdef}
\begin{proof}
We will prove that $\widetilde{\tau}_X(A, Z)$ always exists. 
First we suppose that $A$ is affine,  $\sO_A((1-p^{ne_0})K_A) \cong \sO_A$ and $\mathrm{Hom}_{\sO_A}(F^{e_0}_*\sO_A, \sO_A)$ is generated by $\varphi_{A, e_0}$ as an $F^{e_0}_*\sO_A$-module. 
Then $\mathrm{Hom}_{\sO_A}(F^{ne_0}_*\sO_A, \sO_A)$ is generated by $\varphi_{A, ne_0}$ as an $F^{ne_0}_*\sO_A$-module for all $n \ge 1$. 
Here we use the following claim.

\begin{cl}
There exists an element $\gamma \in \sO_A$ not contained in any minimal prime ideal of $\mathcal{I}_X$ and satisfying the following property:  for every $\delta \in \sO_A$ not contained in any minimal prime of $\mathcal{I}_X$, there exists an integer $n \ge 1$ such that 
$$\gamma \in \varphi_{A, ne_0}(F^{ne_0}_*(\delta \mathcal{I}_X^{c(p^{ne_0}-1)}\mathcal{I}_{Z_1}^{\lceil t_1(p^{ne_0}-1) \rceil} \cdots \mathcal{I}_{Z_m}^{\lceil t_m(p^{ne_0}-1) \rceil})).$$
\end{cl}
\begin{proof}
Suppose that $g \in \bigcap_i \mathcal{I}_{Z_i}$ is an element not contained in any minimal prime of $\mathcal{I}_X$ such that $D(g)\big|_X \subseteq X$ is regular.  By \cite[Example 2.6]{Ta2}, $D(g)$ is purely $F$-regular along $D(g)\big|_X$. 
It then follows from an argument similar to \cite[Proposition 3.21]{Sc3} that some power of $g$ satisfies the condition of the claim. 
\end{proof}

Let $\gamma \in \sO_A$ be an element satisfying the conditions of the above claim. 
Then we will show that 
$$\widetilde{\tau}_X(A, Z)=\sum_{n \ge 0} \varphi_{A, ne_0}(F^{ne_0}_*(\gamma \mathcal{I}_X^{c(p^{ne_0}-1)}\mathcal{I}_{Z_1}^{\lceil t_1(p^{ne_0}-1)\rceil} \cdots \mathcal{I}_{Z_m}^{\lceil t_m(p^{ne_0}-1) \rceil})).$$
It is easy to check that $\sum_{n \ge 0} \varphi_{A, ne_0}(F^{ne_0}_*(\gamma \mathcal{I}_X^{c(p^{ne_0}-1)}\mathcal{I}_{Z_1}^{\lceil t_1(p^{ne_0}-1)\rceil} \cdots \mathcal{I}_{Z_m}^{\lceil t_m(p^{ne_0}-1) \rceil} ))$ is the smallest ideal $J \subseteq \sO_A$ containing $\gamma$ and satisfying 
$$\varphi_{A,ne_0}(F^{ne_0}_*(J\mathcal{I}_X^{c(p^{ne_0}-1)}\mathcal{I}_{Z_1}^{\lceil t_1(p^{ne_0}-1) \rceil} \cdots \mathcal{I}_{Z_m}^{\lceil t_m(p^{ne_0}-1) \rceil})) \subseteq J$$
for all $n \ge 1$. 
On the other hand, if an ideal $I \subseteq \sO_A$ is not contained in any minimal prime of $\mathcal{I}_X$ and satisfying 
$$\varphi_{A,ne_0}(F^{ne_0}_*(I \mathcal{I}_X^{c(p^{ne_0}-1)}\mathcal{I}_{Z_1}^{\lceil t_1(p^{ne_0}-1) \rceil} \cdots \mathcal{I}_{Z_m}^{\lceil t_m(p^{ne_0}-1) \rceil})) \subseteq I$$
for all $n \ge 1$, 
then $\gamma$ is forced to be in $I$ by definition.  
This complete the proof when $A$ is affine and $\mathrm{Hom}_{\sO_A}(F^{e_0}_*\sO_A, \sO_A)$ is generated by $\varphi_{A, e_0}$ as an $F^{e_0}_*\sO_A$-module.

In the general case, $\widetilde{\tau}_X(A, Z)$ is obtained by gluing the constructions on affine charts. 
\end{proof}

\begin{rem}
The definition of $\widetilde{\tau}_X(A,Z)$ is independent of the choice of $e_0$. 
\end{rem}

Next, we will give a definition of $F$-singularities of pairs and $F$-pure thresholds. 

\begin{defn}[\textup{\cite[Definition 3.1]{Ta}, \cite[Proposition 3.3]{Sc}, cf.~\cite[Proposition 5.3]{Sc}}]\label{F-sing}
Let $X$ be an $F$-finite integral normal scheme of characteristic $p>0$ and $D$ be an effective $\Q$-divisor on $X$. 
Let $Z=\sum_{i=1}^m t_i Z_i$ be a formal combination where the $t_i$ are nonnegative real numbers and the $Z_i$ are proper closed subschemes of $X$. 
Fix an arbitrary point $x \in X$. 
\renewcommand{\labelenumi}{(\roman{enumi})}
\begin{enumerate}
\item
$((X,D); Z)$ is said to be \textit{strongly $F$-regular} at $x$ if for every nonzero $\gamma \in \sO_{X, x}$, there exist an integer $e \ge 1$ and $\delta \in \mathcal{I}_{Z_1, x}^{\lceil t_1(p^{e}-1) \rceil} \cdots \mathcal{I}_{Z_m, x}^{\lceil t_m(p^{e}-1) \rceil}$ such that 
$$\gamma \delta F^e: \sO_{X, x} \to F^e_*\sO_{X}(\lceil (p^e-1) D \rceil)_x \quad a \mapsto \gamma \delta a^{p^e}$$ 
splits as an $\sO_{X, x}$-module homomorphism. 
\item 
$((X,D); Z)$ is said to be \textit{sharply $F$-pure} at $x$ if there exist an integer $e \ge 1$ and $\delta \in \mathcal{I}_{Z_1, x}^{\lceil t_1(p^{e}-1) \rceil} \cdots \mathcal{I}_{Z_m, x}^{\lceil t_m(p^{e}-1) \rceil}$ such that 
$$\delta F^e: \sO_{X, x} \to F^e_*\sO_{X}(\lceil (p^e-1)D \rceil)_x \quad a \mapsto \delta a^{p^e}$$ splits as an $\sO_{X, x}$-module homomorphism. 
\end{enumerate}
We simply say that $(X; Z)$ is strongly $F$-regular (resp.~sharply $F$-pure) at $x$ if so is $((X, 0); Z)$.  
We say that $(X, D)$ is strongly $F$-regular (resp.~sharply $F$-pure) if so is $((X, D); \emptyset)$.  
Also, we say that $((X,D); Z)$ is strongly $F$-regular (resp.~sharply $F$-pure) if so is it for all $x \in X$. 
\begin{enumerate}
\item[(iii)]
Suppose that $(X, D)$ is sharply $F$-pure at $x$. 
Then the \textup{$F$-pure threshold} $\mathrm{fpt}_x((X,D); Z)$ of $Z$ at $x$ is defined to be 
$$\quad \mathrm{fpt}_x((X,D); Z):=\sup\{t \in \R_{\ge 0} \mid ((X, D); tZ)\textup{ is sharply $F$-pure at $x$}\}.$$
\end{enumerate}
We denote this threshold by $\mathrm{fpt}_x(X; Z)$ when $D=0$.
\end{defn}

\begin{rem}
Let $A$ and $Z$ be as in Proposition-Definition \ref{test ideal def}. Then $(A; Z)$ is strongly $F$-regular at a point $x \in A$ if and only if $\widetilde{\tau}(A, Z)_x=\sO_{A, x}$. 
\end{rem}

There exists a criterion for sharp $F$-purity, so-called the Fedder type criterion, which we will use later. 

\begin{lem}[\textup{\cite[Lemma 1.6]{Fe}, \cite[Theorem 4.1]{Sc}}]\label{Fedder}
Let $A$ be an $F$-finite regular integral affine scheme of characteristic $p>0$ and $X \subseteq A$ be a reduced equidimensional closed subscheme. 
\begin{enumerate}
\item For each nonnegative integer $e$, the natural morphism 
$$F^e_*(\mathcal{I}_{X}^{[p^e]}:\mathcal{I}_{X}) \cdot  \mathscr{H}\mathrm{om}_{\sO_{A}}(F^e_*\sO_{A}, \sO_{A}) \to \mathscr{H}\mathrm{om}_{\sO_{X}}(F^e_*\sO_{X}, \sO_{X})$$
sending $s \cdot \varphi_A$ to $\overline{\varphi_A \circ F^e_*(\times s)}$ induces the isomorphism 
$$\frac{F^e_*(\mathcal{I}_{X}^{[p^e]}:\mathcal{I}_{X}) \cdot  \mathscr{H}\mathrm{om}_{\sO_{A}}(F^e_*\sO_{A}, \sO_{A})}{F^e_*\mathcal{I}_{X}^{[p^e]} \cdot \mathscr{H}\mathrm{om}_{\sO_{A}}(F^e_*\sO_{A}, \sO_{A})} \cong \mathscr{H}\mathrm{om}_{\sO_{X}}(F^e_*\sO_{X}, \sO_{X}).$$

\item Let $Z=\sum_{i=1}^m t_i Z_i$ be a formal combination where the $t_i$ are nonnegative real numbers and the $Z_i$ are proper closed subschemes of $A$ which do not contain any
component of $X$ in their support. 
Let $x \in X$ be an arbitrary point.   
Then the following conditions are equivalent to each other:
\begin{enumerate}
\item[\textup{(a)}] $(X; Z\big|_X)$ is sharply $F$-pure at $x$, 
\item[\textup{(b)}] there exists an integer $e_0 \ge 1$ such that 
$$(\mathcal{I}_{X, x}^{[p^{e_0}]}:\mathcal{I}_{X,x})\mathcal{I}_{Z_1, x}^{\lceil t_1(p^{e_0}-1) \rceil} \cdots  \mathcal{I}_{Z_m, x}^{\lceil t_m(p^{e_0}-1) \rceil}\not\subseteq \m_{A, x}^{[p^{e_0}]},$$
which is equivalent to saying that 
$$(\mathcal{I}_{X, x}^{[p^{ne_0}]}:\mathcal{I}_{X,x})\mathcal{I}_{Z_1, x}^{\lceil t_1(p^{ne_0}-1) \rceil} \cdots  \mathcal{I}_{Z_m, x}^{\lceil t_m(p^{ne_0}-1) \rceil}\not\subseteq \m_{A, x}^{[p^{ne_0}]}$$
for all integers $n \ge 1$. 
Here, $\m_{A, x} \subseteq \sO_{A, x}$ denotes the maximal ideal of $x$. 
\end{enumerate}
\end{enumerate}
\end{lem}
We remark that (2) is an easy consequence of (1) in Lemma \ref{Fedder}. 


\subsection{Singularities of the minimal model program}
In this subsection, we recall the definitions of adjoint ideal sheaves, multiplier ideal sheaves and singularities of pairs. The reader is referred to \cite{La} for basic theory of multiplier ideal sheaves and to \cite{Ei}, \cite{Ta2} for that of adjoint ideal sheaves. 

Let $X$ be a normal variety over an algebraically closed field $K$ of characteristic zero and let $Z=\sum_i t_i Z_i$ be a formal combination where the $t_i$ are nonnegative real numbers and the $Z_i$ are proper closed subschemes of $X$. 
A \textit{log resolution} of the pair $(X, Z)$ is a proper birational morphism $\pi:\widetilde{X} \to X$ with $X$ a smooth variety such that all scheme theoretic inverse images $\pi^{-1}(Z_i)$ are divisors and in addition $\bigcup_{i} \mathrm{Supp}\; \pi^{-1}(Z_i) \cup \mathrm{Exc}(\pi)$ is a simple normal crossing divisor. 
The existence of log resolutions is guaranteed by Hironaka's desingularization theorem \cite{Hi}.

\begin{defn}\label{adjoint ideal 2}
Let $X$ and $Z$ be as above and let $D=\sum_k d_k D_k$ be a boundary divisor on $X$, that is, $D$ is a $\Q$-divisor on $X$ with $0 \le d_k \le 1$ for all $k$. 
In addition, we assume that  $K_X+D$ is $\Q$-Cartier and no component of $\lfloor D \rfloor$ is contained in the support of the $Z_i$. 
Fix a log resolution $\pi: \widetilde{X} \to X$ of $(X, D+Z)$ such that $\pi^{-1}_*\lfloor D \rfloor$ is smooth. 
Then the \textit{adjoint ideal sheaf} $\adj_D(X, Z)$ of $(X, Z)$ along $D$ is defined to be 
$$\adj_D(X, Z)=\pi_*\sO_{\widetilde{X}}(\lceil K_{\widetilde{X}}-\pi^*(K_X+D)-\sum_i t_i \ \pi^{-1}(Z_i) \rceil+\pi^{-1}_*\lfloor D \rfloor) \subseteq \sO_X. $$
When $D=0$, we denote this ideal sheaf by $\J(X, Z)$ and call it the \textit{multiplier ideal sheaf} associated to $(X, Z)$. 
\end{defn}

\begin{defn}\label{sing of pairs}
Let $X$ and $Z$ be as above and let $D$ be a $\Q$-divisor on $X$ such that $K_X+D$ is $\Q$-Cartier. 
Fix a log resolution $\pi: \widetilde{X} \to X$ of $(X, D+Z)$, and then we can write 
$$K_{\widetilde{X}}=\pi^*(K_X+D)+\sum_{i}t_i \ \pi^{-1}(Z_i)+\sum_j a_j E_j, $$
where the $a_j$ are real numbers and the $E_j$ are prime divisors on $\widetilde{X}$. 
Fix an arbitrary point $x \in X$. 
\renewcommand{\labelenumi}{(\roman{enumi})}
\begin{enumerate}
\item $((X,D);  Z)$ is said to be \textit{klt} at $x$ if $a_j >-1$ for all $j$ such that $x \in \pi(E_j)$. 
\item $((X,D); Z)$ is said to be \textit{log canonical} at $x$ if $a_j \ge -1$ for all $j$ such that $x \in \pi(E_j)$. 
\end{enumerate}
When $X$ is $\Q$-Gorenstein and $D=0$, we simply say that $(X; Z)$ is kit (resp.~log canonical) at $x$ instead of saying that  $((X, 0); Z)$ is kit (resp.~log canonical) at $x$.  
When $Z=\emptyset$, we simply say that $(X, D)$ is kit (resp.~log canonical) at $x$ instead of saying that  $((X, D); \emptyset)$ is kit (resp.~log canonical) at $x$.  
Also, we say that $((X,D); Z)$ is klt (resp.~log canonical) if so is it for all $x \in X$. 
\begin{enumerate}
\item[(iii)] 
Suppose that $(X, D)$ is log canonical at $x$. 
Then the \textit{log canonical threshold} $\lct_x((X,D); Z)$ of $Z$ at $x$ is defined to be
$$\lct_x((X,D); Z):=\sup\{t \in \R_{\ge 0} \mid ((X, D);tZ)\textup{ is log canonical at }x\}.$$
\end{enumerate}
We simply denote this threshold by $\lct_x(X; Z)$ if $X$ is $\Q$-Gorenstein and $D=0$. 
\end{defn}

When the ambient variety is smooth, we can generalize the notion of adjoint ideal sheaves to the higher codimension case. 
Let $A$ be a smooth variety over an algebraically closed field of characteristic zero and $X \subseteq A$ be a reduced equidimensional closed subscheme of codimension $c$. 

\begin{defn}[\textup{\cite[Definition 1.6]{Ta2}, cf.~\cite[Definition 3.4]{Ei}}]\label{adjoint ideal}
Let the notation be as above. 
Let $Z=\sum_i t_i Z_i$ be a  formal combination where the $t_i$ are nonnegative real numbers and the $Z_i$ are proper closed subschemes of $A$ which do not contain any component of $X$ in their support. 
\renewcommand{\labelenumi}{(\roman{enumi})}
\begin{enumerate}
\item 
Let $f:A' \to A$ be the blow-up of $A$ along $X$ and $E$ be the reduced 
exceptional divisor of $f$ that dominates $X$. 
Let $g: \widetilde{A} \to A'$ be a log resolution of $(A', f^{-1}(X)+\sum_i f^{-1}(Z_i))$ so that the strict transform $g^{-1}_*E$ is smooth and set $\pi=f \circ g$.
Then the \textit{adjoint ideal sheaf} $\adj_X(A,Z)$ of the pair $(A,Z)$ along $X$ is defined to be 
$$\adj_X(A,Z):=\pi_*\sO_{\widetilde{A}}(K_{\widetilde{A}/A}-c \ \pi^{-1}(X)-\lfloor \sum_i t_i \ \pi^{-1}(Z_i) \rfloor +g^{-1}_*E).$$

\item $(A; Z)$ is said to be \textit{plt} along $X$ if $\adj_X(A, Z)=\sO_A$. 
\end{enumerate}
\end{defn}

\begin{rem}
(1) 
The above definitions (Definitions \ref{adjoint ideal 2}, \ref{sing of pairs}  and \ref{adjoint ideal}) are independent of the choice of a log resolution used to define them. 

(2) 
Let $X$ and $Z$ be as in Definition \ref{sing of pairs}, and assume that $X$ is $\Q$-Gorenstein. Then $(X; Z)$ is klt at a point $x \in X$ if and only if $\J(X, Z)_x=\sO_{X, x}$. 
\end{rem}


\section{Reduction from characteristic zero to characteristic $p$}
In this section, we briefly review how to reduce things from characteristic zero characteristic $p > 0$. Our main references are \cite[Chapter 2]{HH} and \cite[Section 3.2]{MS}.

Let $X$ be a scheme of finite type over a field $K$ of characteristic zero and $Z=\sum_i t_i Z_i$ be a  formal combination where the $t_i$ are real numbers and the $Z_i$ are proper closed subschemes of $X$. 
Choosing a suitable finitely generated $\Z$-subalgebra $B$ of $K$, we can construct a scheme $X_B$ of finite type over $B$ and closed subschemes $Z_{i, B} \subsetneq X_B$  such that 
there exist isomorphisms
\[\xymatrix{
X \ar[r]^{\cong \hspace*{3em}} &  X_B \times_{\Spec B} \Spec K\\
Z_i \ar[r]^{\cong \hspace*{3em}} \ar@{^{(}->}[u] & Z_{i, B} \times_{\Spec B} \Spec K. \ar@{^{(}->}[u]\\
}\]
Note that we can enlarge $B$ by localizing at a single nonzero element and replacing $X_B$ and $Z_{i,B}$ with the corresponding open subschemes. 
Thus, applying the generic freeness \cite[(2.1.4)]{HH}, we may assume that $X_B$ and the $Z_{i, B}$ are flat over $\Spec B$.
Letting $Z_B:=\sum_i t_i Z_{i,B}$, we refer to $(X_B, Z_B)$ as a \textit{model} of $(X, Z)$ over $B$.   
Given a closed point $\mu \in \Spec B$, we denote by $X_{\mu}$ (resp.~$Z_{i, \mu}$) the fiber of $X_B$ (resp.~$Z_{i, B}$) over $\mu$ and denote $Z_{\mu}:=\sum_i t_i Z_{i, \mu}$.  
Then $X_{\mu}$ is a scheme of finite type over the residue field $\kappa(\mu)$ of $\mu$, which is a finite field of characteristic $p(\mu)$.  
If $X$ is regular, then after possibly enlarging $B$, we may assume that $X_B$ is regular. In particular, there exists a dense open subset $W \subseteq \Spec B$ such that $X_{\mu}$ is regular for all closed points $\mu \in W$. 
Similarly, if $X$ is normal (resp.~reduced, irreducible, locally a complete intersection, Gorenstein, $\Q$-Gorenstein of index $r$, Cohen-Macaulay), then so is $X_{\mu}$ for general closed points $\mu \in \Spec B$. 
Also, $\dim X=\dim X_{\mu}$ and $\codim(Z_i, X)=\codim(Z_{i,\mu}, {X_{\mu}})$ for general closed points $\mu \in \Spec B$. 
In particular, if $X$ is normal and $Z$ is an $\Q$-Weil (resp.~$\Q$-Cartier) divisor on $X$, then $Z_{\mu}$ is an $\Q$-Weil (resp.~$\Q$-Cartier) divisor on $X_{\mu}$ for general closed points $\mu \in \Spec B$. 
If $K_X$ is a canonical divisor on $X$, then $K_{X, \mu}$ gives a canonical divisor $K_{X_{\mu}}$ on $X_{\mu}$ for general closed points $\mu \in \Spec B$. 

Given a morphism $f:X \to Y$ of schemes of finite type over $K$ and a model $(X_B, Y_B)$ of $(X, Y)$ over $B$,  after possibly enlarging $B$, we may assume that $f$ is induced by a morphism $f_B :X_B \to Y_B$ of schemes of finite type over $B$. 
Given a closed point $\mu \in \Spec B$, we obtain a corresponding morphism $f_{\mu}:X_{\mu} \to Y_{\mu}$ of schemes of finite type over $\kappa(\mu)$. 
If $f$ is projective (resp.~finite), then so is $f_{\mu}$ for general closed points $\mu \in \Spec B$. 

\begin{defn}
Let $\mathbf{P}$ be a property defined for a triple $(X,D,Z)$ where $X$ is a scheme of finite type over a finite field, $D$ is an effective $\Q$-divisor on $X$ and $Z$ is an $\R_{\ge 0}$-linear combination of closed subschemes of $X$. 
\renewcommand{\labelenumi}{(\roman{enumi})}
\begin{enumerate}
\item $((X,D);Z)$ is said to be of $\mathbf{P}$ \textit{type} if for a model of $(X, D,Z)$ over a finitely generated $\Z$-subalgebra $B$ of $K$, there exists a dense open subset $W \subseteq \Spec B$ such that $((X_{\mu}, D_{\mu}); Z_{\mu})$ satisfies $\mathbf{P}$ for all closed points $\mu \in W$. 
\item $((X,D); Z)$ is said to be of \textit{dense $\mathbf{P}$ type} if for a model of $(X, D, Z)$ over a finitely generated $\Z$-subalgebra $B$ of $K$, there exists a dense subset of closed points $W \subseteq \Spec B$ such that $((X_{\mu}, D_{\mu}); Z_{\mu})$ satisfies $\mathbf{P}$ for all $\mu \in W$. 
\end{enumerate}
\end{defn}

\begin{rem}
(1) By enlarging $B$, $((X,D); Z)$ is of $\mathbf{P}$ type if and only if for some model over $B$, $\mathbf{P}$ holds for all closed points $\mu \in \Spec B$. 

(2) When $\mathbf{P}$ is strong $F$-regularity,  pure $F$-regularity, or sharp $F$-purity, the above definition is independent of the choice of a model. 
\end{rem}

There exists a correspondence between adjoint ideal sheaves and test ideal sheaves.  

\begin{thm}[\textup{\cite[Theorem 5.3]{Ta4}, cf.~\cite{HY}, \cite{Ta5}}]\label{mult thm}
Let $X$ be a normal variety over a field $K$ of characteristic zero and let $Z=\sum_{i} t_i Z_i$ be a formal combination where the $t_i$ are nonnegative real numbers and the $Z_i$ are proper closed subschemes of $X$. 
Let $D=\sum_j d_j D_j$ be a boundary divisor on $X$ such that  $K_X+D$ is $\Q$-Cartier and no component of $\lfloor D \rfloor$ is contained in the support of the $Z_i$. 
Given any model of $(X, Z, D)$ over a finitely generated $\Z$-subalgebra $B$ of $K$, 
there exists a dense open subset $W \subseteq \Spec B$ such that  
$$\adj_D(X, Z)_{\mu}=\widetilde{\tau}_{D_{\mu}}(X_{\mu},Z_{\mu})$$
for every closed point $\mu \in W$. 
In particular, 
$((X, D); Z)$ is klt at $x$ if and only if it is of strongly $F$-regular type at $x$. 
\end{thm}

An analogous correspondence between log canonicity and $F$-purity, that is, the equivalence of log canonical pairs and pairs of dense sharply $F$-pure type is largely conjectural.  
\begin{conj}\label{lc conj}
Let $X$ be a normal variety over an algebraically closed field $K$ of characteristic zero and $D$ be an effective $\Q$-divisor on $X$ such that $K_X+D$ is $\Q$-Cartier. 
Let $Z=\sum_i t_i Z_i$ be a  formal combination where the $t_i$ are nonnegative rational numbers and the $Z_i$ are proper closed subschemes of $X$. 
Fix an arbitrary point $x \in X$. 
\begin{enumerate}
\item
$((X, D); Z)$ is log canonical at $x$ if and only if it is of dense sharply $F$-pure type at $x$. 
\item
Suppose that $(X, D)$ is log canonical at $x$.
Given any model of $(X, D, Z, x)$ over a finitely generated $\Z$-subalgebra $B$ of $K$, 
there exists a dense subset of closed points $W \subseteq \Spec B$ such that  
$$\lct_x((X, D); Z)=\fpt_{x_{\mu}}((X_{\mu}, D_{\mu}); Z_{\mu})$$
for all $\mu \in W$. 
\end{enumerate}
\end{conj}

\begin{rem}\label{F-pure=>lc}
(1) It is easy to see that (1) implies (2) in Conjecture \ref{lc conj}. 

(2) If $((X, D); Z)$ is of dense sharply $F$-pure type at $x$, then by \cite[Theorem 3.3]{HW} and \cite[Proposition 3.8]{Ta3}, it is log canonical at $x$. 
\end{rem}

\begin{rem}\label{known cases}
Conjecture \ref{lc conj} is known to hold in the following cases (see also Theorem \ref{very general}):
\begin{enumerate}[(i)]
\item 
$X$ is a $\Q$-Gorenstein toric variety, $D=0$ and the $Z_i$ are monomial subschemes. 
\item 
$X$ is the affine space $\mathbb{A}^n_K$, $D=0$ and $Z=t_1 Z_1$ where $Z_1$ is a binomial complete intersection subscheme or a space monomial curve (in the latter case, $n=3$). 
\item
$X$ is a normal surface, $D$ is an integral effective divisor on $X$ and $Z=\emptyset$. 
\item 
$X$ is the affine space $\mathbb{A}^n_K$, $D=0$ and $Z$ is a hypersurface of $X$ such that the coefficients of terms of its defining equation are algebraically independent over $\Q$. 

\end{enumerate}
The case (i) follows from \cite[Theorem 3]{Bl}, the case (ii) does from \cite[Theorem 0.1]{ST} and the case (iv) does from \cite[Theorem 5.16]{He}. 
We explain here how to check the case (iii). 
If $D \neq 0$, then it follows from comparing \cite[Theorem 4.5]{HW} with \cite[Theorem 9.6]{Kaw}. 
So we consider the case where $D =0$. 
By Remark \ref{F-pure=>lc}, it suffices to show that a two-dimensional log canonical singularity $(X, x)$ is of dense $F$-pure type. 
Passing to an index one cover, we may assume that $(X, x)$ is Gorenstein. If it is log terminal, then  by \cite[Theorem 5.2]{Ha} (see also Theorem \ref{mult thm}),  it is of $F$-regular type and, in particular, of dense $F$-pure type.  
Hence we can assume that $(X, x)$ is not log terminal, that is, $(X, x)$ is a cusp singularity or a simple elliptic singularity. 
By \cite[Theorem 1.2]{MeS} or \cite[Theorem 1.7]{Wa}, cusp singularities are of dense $F$-pure type. 
Also, by \cite{MeS}, a simple elliptic singularity with exceptional elliptic curve $E$ is of dense $F$-pure type if and only if for a model $E_B$ of $E$ over a finitely generated $\Z$-subalgebra $B \subseteq K$, there exists a dense subset of closed points $W \subseteq \Spec B$ such that $E_{\mu}$ is ordinary for all $\mu \in W$.  
Applying the same argument as the proof of \cite[Proposition 5.3]{MS}, we may assume that $E$ is defined over $\overline{\Q}$. 
It then follows from Serre's ordinary reduction theorem \cite{Se} that such $W$ always exists. 
Thus, simple elliptic singularities are of dense $F$-pure type. 
\end{rem}

\begin{lem}\label{Cartier}
In order to prove Conjecture \ref{lc conj}, it is enough to consider the case when $Z=\emptyset$. 
\end{lem}

\begin{proof}
Since the question is local, we work in a sufficiently small neighborhood of $x$.
By Remark 2.5, it suffices to show that if $((X,D); Z)$ is log canonical, then it is of dense sharply $F$-pure type.

Suppose that $((X,D); Z)$ is log canonical. 
Let $h_{i, 1}, \dots, h_{i, m_i}$ be a system of generators for $\mathcal{I}_{Z_i}$ for each $i$. 
Let $g_{i, 1}, \dots, g_{i,m_i}$ be general linear combinations of $h_{i, 1}, \dots, h_{i, m_i}$ with coefficients in $K$, and set $g_i=\prod_{j=1}^{m_i}g_{i,j}$, so that 
$(X, D+\sum_{i} \frac{t_i}m_i \mathrm{div}_X(g_i))$ is log canonical. 
On the other hand, since $g_i \in \mathcal{I}_{Z_i}^{m_i}$, if $(X, D+\sum_{i} \frac{t_i}{m_i} \mathrm{div}_X(g_i))$ is of dense sharply $F$-pure type, then so is $((X,D); Z)$. 
Therefore, it is enough to show that the log canonical pair $(X, D+\sum_{i} \frac{t_i}{m_i} \mathrm{div}_X(g_i))$ is of dense sharply $F$-pure type. 
\end{proof}

Musta\c{t}\u{a} and Srinivas \cite{MS} recently proposed the following more arithmetic conjecture and related it to another conjecture on a comparison between multiplier ideal sheaves and test ideal sheaves. 

\begin{conj}[\textup{\cite[Conjecture 1.1]{MS}}]\label{MS conj}
Let $X$ be an $n$-dimensional smooth projective variety over $\overline{\Q}$. 
Given a model of $X$ over a finitely generated $\Z$-subalgebra $B$ of $\overline{\Q}$, there exists a dense subset of closed points $W \subseteq \Spec B$ such that the action induced by Frobenius on $H^n(X_{\mu}, \sO_{X_{\mu}})$ is bijective for all $\mu \in W$. 
\end{conj}

\begin{rem}
Conjecture \ref{MS conj} is known to be true when $X$ is a smooth projective curve of genus less than or equal to two (see \cite[Example 5.5]{MS}, which can be traced back to \cite{Og}, \cite{Se}) or a smooth projective surface of Kodaira dimension zero (see \cite[Proposition 2.3]{Ja}). 
\end{rem}

\begin{eg}
We check that Conjecture \ref{MS conj} holds for the Fermat hypersurface $X$ of degree $d$ in $\mathbb{P}^n_{K}$ over a field $K$ of characteristic zero. 
Given a prime number $p$, set $S_p=\F_p[x_0, \dots, x_n]$, $\m_p=(x_0, \dots, x_n) \subseteq S_p$, $f_p=x_0^d+\cdots+x_n^d \in S_p$ and $X_p=\Proj S_p/f_p$.  
Since $H^{n-1}(X_p, \sO_{X_p})=0$ for almost all $p$ when $d \le n$, we consider the case when $d \ge n+1$. 
Note that 
$$H^{n-1}(X_p, \sO_{X_p}) \cong \left[H^{n}_{\m_p}(S_p/f_p)\right]_0 \cong \left[(0:f_p)_{H^{n+1}_{\m_p}(S_p)}\right]_{-d}.$$
Via this isomorphism, the action induced by Frobenius on $H^{n-1}(X_p, \sO_{X_p})$ is identified with 
$$f_p^{p-1}F: \left[(0:f_p)_{H^{n+1}_{\m_p}(S_p)}\right]_{-d} \to \left[(0:f_p)_{H^{n+1}_{\m_p}(S_p)}\right]_{-d}, $$
where $F:H^{n+1}_{\m_p}(S_p) \to H^{n+1}_{\m_p}(S_p)$ is the map induced by Frobenius on $H^{n+1}_{\m_p}(S_p)$. 
Let $\xi=\left[\frac{z}{(x_0\cdots x_n)^m}\right] \in H^{n+1}_{\m_p}(S_p)$ be a homogeneous  element such that $f^{p-1}_pF(\xi)=0$, that is, $f_p^{p-1}z^p \in (x_0^{mp}, \dots, x_n^{mp})$. 
Set $W=\{p \in \Spec \Z \mid p \equiv 1 \mod d \}$, which is a dense subset of $\Spec \Z$, and suppose that $p \in W$. 
Let $a_0, \dots, a_n$ be nonnegative integers such that $\sum_{i=0}^n a_n=d-n-1$. 
Then the term 
$$(x_0^{d(a_0+1)} \cdots x_n^{d(a_n+1)})^{\frac{p-1}{d}}=(x_1^{a_1} \cdots x_n^{a_n})^{p}x_1^{p-a_1-1} \cdots x_n^{p-a_n-1}$$ appears in the expansion of $f_p^{p-1}$. 
Since $\{x_1^{i_1} \cdots x_n^{i_n}\}_{0 \le i_1, \dots, i_n \le p-1}$ is a free basis of $S_p$ as an $S_p^p$-module, $f^{p-1}z^p$ can be written as 
$$f^{p-1}z^p= u (x_1^{a_1} \cdots x_n^{a_n}z)^{p}x_1^{p-a_1-1} \cdots x_n^{p-a_n-1}+\sum_{i_j \ne p-a_j-1}g_{i_1, \cdots, i_n}^p x_1^{i_1} \cdots x_n^{i_n},$$
where $u \in \F_p$ is a nonzero element and $g_{i_1, \cdots, i_n} \in S_p$ for each $0 \le i_1, \dots, i_n \le p-1$. 
Let $\varphi:F_*S_p \to S_p$ be the $S$-linear map sending $x_1^{p-a_1-1} \cdots x_n^{p-a_n-1}$ to $1$ and the other part of the basis to zero. Then 
$$u x_1^{a_1} \cdots x_n^{a_n}z=\varphi(f^{p-1}z^p) \in \varphi( (x_0^{mp}, \dots, x_n^{mp})) \subseteq (x_0^m, \dots, x_n^m).$$
By the definition of the $a_i$, 
one has 
$\m_p^{d-n-1}z \subseteq (x_0^m, \dots, x_n^m)$, that is, $\m_p^{d-n-1}\xi=0$ in $H^{n+1}_{\m_p}(S_p)$. 
This means that 
$\deg \xi \ge -d+1$, and we conclude that $f_p^{p-1}F: \left[(0:f_p)_{H^{n+1}_{\m_p}(S_p)}\right]_{-d} \to \left[(0:f_p)_{H^{n+1}_{\m_p}(S_p)}\right]_{-d}$ is injective for all $p \in W$. 
\end{eg}

The following result comes from a discussion with Karl Schwede, who the author thanks. 

\begin{thm}\label{implication}
If Conjecture \ref{MS conj} holds, then Conjecture \ref{lc conj} holds as well. 
\end{thm}

In order to prove Theorem \ref{implication}, we introduce a notion of sharp $F$-purity for non-effective integral divisors.

\begin{defn}\label{non-effective case}
Let $X$ be an $F$-finite normal integral scheme of characteristic $p>0$ and $D$ be a (not necessarily effective) integral divisor on $X$. 
We assume that $K_X+D$ is $\Q$-Cartier with index not divisible by $p$.  
Let $x \in X$ be an arbitrary point. 
We decompose $D$ as $D=D_+-D_-$, where $D_+$ and $D_-$ are effective integral divisors on $X$ which have no common irreducible components. We then say that the pair $(X, D)$ is \textit{sharply $F$-pure} at $x$ if there exists an integer $e _0 >0$ such that for all positive multiples $e=ne_0$ of $e_0$, one has an $\sO_{X,x}$-linear map $\varphi : F^e_*\sO_X((p^e-1)D_++D_-)_x \to \sO_X(D_-)_x$ whose image of  $F^e_*\sO_X(D_-)_x$ contains $1$. 
We say that $(X, D)$ is sharply $F$-pure if it is sharply $F$-pure at every closed point of $X$. 
\end{defn}

If $D$ is an effective integral divisor, then this definition coincides with Definition \ref{F-sing} (ii). 
We need a variant of \cite[Theorem 6.28]{ScT} involving sharp $F$-purity in the sense of Definition \ref{non-effective case}. 
\begin{lem}[\textup{cf.~\cite[Theorem 6.28]{ScT}}]\label{finite map}
Let $\pi:Y \to X$ be a finite separable morphism of $F$-finite normal integral schemes of characteristic $p>0$.
Let $\Delta_X$ be an effective $\Q$-divisor on $X$ such that $K_X+\Delta_X$ is $\Q$-Cartier with index not divisible by $p$. 
Suppose that $\Delta_Y$ is an integral divisor on $Y$ such that $K_Y+\Delta_Y=\pi^*(K_X+\Delta_X)$. 
In addition, we assume that the trace map $\mathrm{Tr}_{Y/X}:\pi_*\sO_Y \to \sO_X$ is surjective. 
Then $(X, \Delta_X)$ is sharply $F$-pure if and only if $(Y, \Delta_Y)$ is sharply $F$-pure in the sense of Definition \ref{non-effective case}.
\end{lem}

\begin{proof}
The statement is local on $X$ and so we assume that $X =\Spec A$ and $Y=\Spec B$, where $A$ is a local ring and $B$ is a semi-local ring. 
There exists $e_0 \in \N$ such that $(p^{e_0}-1)(K_X+\Delta)$ is Cartier. 
Then $\Hom_A(F^e_*A((p^e-1)\Delta_X ), A)$ is a free $F^e_*A$-module of rank one for all positive multiples $e = ne_0$ of $e_0$. 
Let $\varphi_X: F^e_*A \to A$ be its generator. 
We decompose $\Delta_Y$ as $\Delta_{Y,+}-\Delta_{Y,-}$, where $\Delta_{Y,+}, \Delta_{Y,-}$ are effective integral divisors on $Y$ which have no common components. 
Then the $F^e_*B$-module
\begin{align*}
\Hom_B(F^e_*B((p^e-1)\Delta_{Y,+}+\Delta_{Y,-}), B(\Delta_{Y,-})) &\cong F^e_*B((1-p^e)(K_Y+\Delta_Y))\\
&= F^e_*\pi^*A((1-p^e)(K_X+\Delta_X))\\
&\cong F^e _*\pi^*A=F^e_*B,
\end{align*}
and we pick its generator $\varphi_Y:F^e_*B(\Delta_{Y,-}) \to B(\Delta_{Y,-})$ extending $\varphi_X: F^e_*A \to A$. 

Suppose that $(X, \Delta_X)$ is sharply $F$-pure. 
By the definition of sharp $F$-purity, after possibly enlarging $e$, we have that $1 \in \mathrm{Im}\; \varphi_X \subseteq \mathrm{Im}\; \varphi_Y$, and hence $(Y, \Delta_Y)$ is sharply $F$-pure. 

Conversely, suppose that $(Y, \Delta_Y)$ is sharply $F$-pure. 
Making $e$ larger if necessary, we may assume that $1 \in \mathrm{Im}\; \varphi_Y$. 
Note that $\Delta_Y=\pi^*\Delta_X-R$ and $R \ge \Delta_{Y,-}$, where $R$ denotes the ramification divisor of $\pi$. 
Then the $F^e_*B$-module
\begin{align*}
& \Hom_B(F^e_*B((p^e-1)\pi^*\Delta_X+R), B(R))\\
\cong & F^e_*\pi^*A((1-p^e)(K_X+\Delta_X))\\
\cong & \Hom_B(F^e_*B((p^e-1)\Delta_{Y,+}+\Delta_{Y,-}), B(\Delta_{Y,-})),
\end{align*}
and we pick its generator $\widetilde{\varphi}_Y:F^e_*B(R) \to B(R)$ extending $\varphi_Y: F^e_*B(\Delta_{Y,-}) \to B(\Delta_{Y,-})$. 
Since the trace map $\mathrm{Tr}_{Y/X}$ corresponds to the ramification divisor $R$, 
we have the following commutative diagram:
$$
\xymatrix{
F^e_*A \ar[r]^{\varphi_X}  & A  \\
F^e_*B(R) \ar[r]^{\widetilde{\varphi}_Y} \ar[u]^{F^e_*\mathrm{Tr}_{Y/X}} & B(R) \ar[u]_{\mathrm{Tr}_{Y/X}}.\\
}
$$
The surjectivity of the trace map $\mathrm{Tr}_{Y/X}:B  \to A$ implies that 
$$1 \in \mathrm{Tr}_{Y/X}(\mathrm{Im}\; \varphi_{Y}) \subseteq \mathrm{Tr}_{Y/X}(\mathrm{Im}\; \widetilde{\varphi}_Y)=\varphi_X( \mathrm{Im}\;  F^e_*\mathrm{Tr}_{Y/X})=\mathrm{Im}\; \varphi_X, $$
because $B \subseteq \mathrm{Im}\; \varphi_Y$. 
Thus, $(X, \Delta_X)$ is sharply $F$-pure. 
\end{proof}

\begin{proof}[Proof of Theorem \ref{implication}]
Let the notation be as in Conjecture \ref{lc conj}.
By Lemma \ref{Cartier}, we may assume that $K_X+D$ is Cartier and $Z=\emptyset$.  
Since the question is local, we work in a sufficiently small neighborhood of  $x$. 
By Remark \ref{F-pure=>lc}, it suffices to show that if $(X, D)$ is log canonical, then it is of dense sharply $F$-pure type. 

Suppose that $(X,D)$ is log canonical. 
By \cite[Section 2.4]{KM}, there exists a finite morphism $f: X' \to X$ from a normal variety $X'$ over $K$ such that $f^*(K_X+D)$ is Cartier. 
Let $D'$ be a (not necessarily effective) integral divisor on $X'$ such that $K_{X'}+D'=f^*(K_X+D)$. 
It then follows from \cite[Proposition 5.20]{KM} that $(X', D')$ is log canonical. 
We decompose $D'$ as $D'=D'_+-D'_-$, where $D'_+$ and $D'_-$ are effective integral divisors on $X$ which have no common components. 
Take a log resolution $\pi: \widetilde{X} \to X'$ of $(X', D')$, and denote by $E$ the reduced divisor supported on the $\pi$-exceptional locus $\mathrm{Exc}(\pi)$. 
 Let $(X_B, D_B,  X'_B, D'_B=D'_{+,B}-D'_{-, B}, \pi_B, E_{B})$ be a model of $(X, D, X', D'=D'_+-D'_-, \pi, E)$ over a finitely generated $\Z$-subalgebra $B$ of $K$. 
After possibly enlarging $B$, we may assume that $K_{X_{\mu}}+D_{\mu}$ is $\Q$-Cartier with index not divisible by the characteristic $p(\mu)$ and that the trace map $\mathrm{Tr}_{X'_{\mu}/X_{\mu}}: {f_{\mu}}_*\sO_{X'_{\mu}} \to \sO_{X_{\mu}}$ is surjective for all closed points $\mu \in \Spec B$. 

By virtue of \cite[Theorem 5.10]{MS}, there exists a dense subset of closed points $W \subseteq \Spec B$ such that for every integer $e \ge 1$ and every $\mu \in W$, the map
\begin{equation*}\tag{$\diamond$}
{\pi_{\mu}}_*F^e_*(\sO_{\widetilde{X}_{\mu}}(K_{\widetilde{X}_{\mu}}+{\pi_{\mu}}^{-1}_*{D'_{+, \mu}}+E_{\mu})) \to {\pi_{\mu}}_*\sO_{\widetilde{X}_{\mu}}(K_{\widetilde{X}_{\mu}}+{\pi_{\mu}}^{-1}_*{D'_{+, \mu}}+E_{\mu}),
\end{equation*}
induced by the canonical dual of the $e$-times iterated Frobenius map $\sO_{\widetilde{X}_{\mu}} \to F^e_* \sO_{\widetilde{X}_{\mu}}$, is surjective. 
Tensoring $(\diamond)$ with $\sO_{X'_{\mu}}(-K_{X'_{\mu}}-D'_{\mu})$, one can see that the map
$$\rho: {\pi_{\mu}}_*F^e_*(\sO_{\widetilde{X}_{\mu}}(M+(1-p(\mu)^e)\pi^*_{\mu}(K_{{X'_{\mu}}}+D'_{\mu}))) \to {\pi_{\mu}}_*\sO_{\widetilde{X}_{\mu}}(M)$$
is surjective, where $M=K_{\widetilde{X_{\mu}}}+{\pi_{\mu}}^{-1}_*{D'_{+,\mu}}-\pi_{\mu}^*(K_{X'_{\mu}}+D'_{\mu})+E_{\mu}$.  
Since $(X', D')$ is log canonical,  $1 \in {\pi_{\mu}}_*\sO_{\widetilde{X}_{\mu}}(M) \subseteq \sO_{X'_{\mu}}(D'_{-, \mu})$. 
By Grothendieck duality, $\rho$ is identified with the evaluation map 
\begin{align*}
F^e_*\sO_{X'_{\mu}}(D'_{-, \mu}) \otimes \mathscr{H}\mathrm{om}_{\sO_{X'_{\mu}}}(F^e_*\sO_{X'_{\mu}}((p(\mu)^e-1)D'_{+,\mu}+D'_{-,\mu}), \sO_{X'_{\mu}}(D'_{-,\mu})) \hspace{4em} \\
\to \sO_{X'_{\mu}}(D'_{-, \mu}).
\end{align*}
The subjectivity of $\rho$ then implies that there exists an $\sO_{X'}$-linear map 
$$\varphi_{X'}:F^e_*\sO_{X'_{\mu}}((p(\mu)^e-1)D'_{+,\mu}+D'_{-,\mu}) \to \sO_{X'_{\mu}}(D'_{-,\mu})$$ such that $1 \in \varphi_{X'}(F^e_*\sO_{X'_{\mu}}(D'_{-, \mu}) )$. 
That is, $(X'_{\mu}, D'_{\mu})$ is sharply $F$-pure in the sense of Definition \ref{non-effective case}. 
Applying Lemma \ref{finite map}, we conclude that $(X_{\mu}, D_{\mu})$ is sharply $F$-pure for all $\mu \in W$. 
\end{proof}

\begin{rem}
Let $Y$ be  an $S2$, $G1$ and seminormal variety over an algebraically closed field $K$ of characteristic zero and $\Gamma$ be an effective $\Q$-Weil divisorial sheaf on $Y$ such that $K_Y+\Gamma$ is $\Q$-Cartier. 
Combining Theorem \ref{implication} with \cite[Corollary 4.4]{MiS}, we can conclude that if Conjecture \ref{MS conj} holds, then the pair $(Y, \Gamma)$ is semi-log canonical if and only if it is of dense sharply $F$-pure type. 
\end{rem}


\section{Restriction theorem for adjoint ideal sheaves}
In this section, building on an earlier work \cite{Ta2}, we give a new proof of Eisenstein's restriction theorem for adjoint ideal sheaves using test ideal sheaves. 

\begin{defn}\label{lci ideal}
Let $A$ be a smooth variety over an algebraically closed field $K$ of characteristic
zero and $X \subseteq A$ be a normal $\Q$-Gorenstein closed subvariety of codimension $c$. 
Denote by $r$ the Gorenstein index of $X$, that is, the smallest positive integer $m$ such that $m K_X$ is Cartier. 
Then the \textit{l.c.i.~defect ideal sheaf}\footnote{We follow a construction due to Kawakita \cite{Ka}, but our terminology is slightly different from his. We warn the reader that the ideal sheaf called the l.c.i.~defect ideal in \cite{Ka} is different from our $J_X$. 
Also, Ein and Musta\c{t}\u{a} \cite{EM} introduced a very similar ideal, which coincides with our $J_X$ up to integral closure.} $J_X \subseteq \sO_X$ is defined as follows:
since the construction is local, we may consider the germ at a closed point $x \in X \subseteq A$. 
We take generally a closed subscheme $Y$ of $A$ which contains $X$ and is locally a complete intersection (l.c.i.~for short) of codimension $c$.
By Bertini's theorem, $Y$ is the scheme-theoretic union of $X$ and another variety $C^Y$ of codimension $c$. 
Then the closed subscheme $D^Y:=C^Y\big|_X$ of $X$ is a Weil divisor such that $rD^Y$ is Cartier and $\sO_X(rK_X)=\sO_X(-rD^Y)\omega_Y^{\otimes r}$. 
The l.c.i.~defect ideal sheaf $J_X$ is defined by
$$J_X=\sum_Y \sO_X(-rD^Y),$$
where $Y$ runs through all the general l.c.i.~closed subschemes of codimension $c$ containing $X$. 
Note that the support of $J_X$ exactly coincides with the non-l.c.i.~locus of $X$. 
In particular, $J_X=\sO_X$ if and only if $X$ is l.c.i.
The reader is referred to \cite[Section 2]{Ka} and \cite[Section 9.2]{EM} for further properties of l.c.i.~defect ideal sheaves.
\end{defn}

Now we give a new proof of Eisenstein's theorem \cite[Corollary 5.2]{Ei}.
\begin{thm}\label{restriction}
Let $A$ be a smooth variety over an algebraically closed field $K$ of characteristic zero and $Z=\sum_{i=1}^m t_i Z_i$ be a formal combination where the $t_i$ are nonnegative real numbers and the $Z_i$ are proper closed subschemes of $A$. 
If $X$ is a normal $\Q$-Gorenstein closed subvariety of $A$ which is not contained in the support of any $Z_i$, then
$$\J(X,Z\big|_X+\frac{1}{r} V(J_X))=\adj_X(A,Z) \big|_X,$$
where $r$ is the Gorenstein index of $X$ and $J_X$ is the l.c.i.~defect ideal sheaf of $X$.  
\end{thm}

\begin{proof}
The proof is a refinement of the proof of \cite[Theorem 3.1]{Ta2}. 
The inclusion $\J(X,Z\big|_X+\frac{1}{r}  V(J_X)) \supseteq \adj_X(A,Z) \big|_X$ follows from a combination of \cite[Remark 8.5]{EM} and \cite[Lemma 1.7]{Ta2}. 
Hence we will prove the converse inclusion. 

Since the question is local, we consider the germ at a closed point $x \in X \cap \bigcap_{i=1}^m Z_i \subset A$. 
Denote by $c$ the codimension of $X$ in $A$. Take generally a subscheme $Y$ of $A$ which contains $X$ and is l.c.i.~of codimension $c$, so $Y$ is the scheme-theoretic union of $X$ and a variety $C^Y$. 
Then $D^Y:=C^Y\big|_ X$ is a Weil divisor on $X$ such that $rD^Y$ is Cartier. 
By a general choice of $Y$, one has 
\begin{equation*}\tag{$\star$}
\J(X,Z\big|_X+\frac{1}{r}  V(J_X))=\adj_{D^Y}(X, Z \big|_X),
\end{equation*}
(which follows from an argument similar to the claim in the proof of \cite[Theorem 3.1]{Ta2}). 
Therefore, it is enough to show that 
$$\adj_{D^Y}(X, Z \big|_X) \subseteq \adj_X(A,Z) \big|_X.$$

By Theorem \ref{mult thm} and \cite[Theorem 2.7]{Ta2}, in order to prove this inclusion, it suffices to show that given any model of $(A, X, Y, Z, C^Y, D^Y)$ over a finitely generated $\Z$-subalgebra $B$ of $K$, one has 
\begin{equation*}\tag{$\star\star$}
\widetilde{\tau}_{D^Y_{\mu}}(X_{\mu}, Z_{\mu} \big|_{X_{\mu}}) \subseteq \widetilde{\tau}_{X_{\mu}}(A_{\mu},Z_{\mu}) \big|_{X_{\mu}}
\end{equation*}
for general closed points $\mu \in \Spec B$. 
Since $\mu$ is a general point of $\Spec B$ and the formation of test ideal sheaves commutes with localization, we may assume that $\sO_{A_{\mu}}$ is an $F$-finite regular local ring of characteristic $p=p(\mu)>r$, $X_{\mu}=V(I)$ is a normal $\Q$-Gorenstein closed subscheme of $A_{\mu}$ with Gorenstein index $r$ and $Y_{\mu}=V((f_1, \dots, f_c))$ is a complete intersection closed subscheme of codimension $c$ containing $X_{\mu}$. 
We may assume in addition that  
$D^Y_{\mu}$ is a Weil divisor on $X_{\mu}$ such that $rD^Y_{\mu}$ is Cartier and $\sO_{X_{\mu}}(rK_{X_{\mu}})=\sO_{X_{\mu}}(-rD^Y_{\mu})\omega_{Y_{\mu}}^{\otimes r}$. 
We take a germ $g \in \sO_{A_{\mu}}$ whose image $\overline{g}$ is the local equation of $r D^Y_{\mu}$ on $\sO_{X_{\mu}}$. 
Let $\ba_i \subseteq \sO_{A_{\mu}}$ be the defining ideal of $Z_{i, \mu}$ for each $i=1, \dots, m$. 
Fix an integer $e_0  \ge 1$ such that $p^{e_0}-1$ is divisible by $r$ and set $q_0=p^{e_0}$. 

\begin{cl} 
For all powers $q=q_0^n$ of $q_0$, one has 
$$g^{(q-1)/r}(I^{[q]}:I)= (f_1 \cdots f_c)^{q-1} \textup{ in $\sO_{A_{\mu}}/I^{[q]}$.}$$
\end{cl}
\begin{proof}[Proof of Claim]
Since $q-1$ is divisible by $r$, 
\begin{align*}
\sO_{X_{\mu}}((1-q)(K_{X_{\mu}}+D^Y_{\mu}))&=\sO_{Y_{\mu}}\left((1-q)K_{Y_{\mu}}\right)\big|_{X_{\mu}}\\
&=\sO_{A_{\mu}}((1-q)(K_{A_{\mu}}+\sum^c_{i=1}\mathrm{div}_{A_{\mu}}(f_i)))\big|_{X_{\mu}}.
\end{align*}
Set $e=ne_0$. By making use of Grothendieck duality, this implies that the natural map of $F^e_*\sO_{A_{\mu}}$-modules
$$\mathrm{Hom}_{\sO_{A_{\mu}}}(F^e_*\sO_{A_{\mu}}((q-1)\sum_{i=1}^c \mathrm{div}_{A_{\mu}}(f_i)), \sO_{A_{\mu}}) \to \mathrm{Hom}_{\sO_{X_{\mu}}}\left(F^e_*\sO_{X_{\mu}}((q-1)D^Y_{\mu}), \sO_{X_{\mu}}\right)$$
induced by restriction is surjective. 
It then follows from Lemma \ref{Fedder} (1) that the $\sO_{A_{\mu}}$-linear map
$$(f_1 \cdots f_c)^{q-1}\sO_{A_{\mu}} \to \frac{{g}^{(q-1)/{r}}(I^{[q]}:I)}{I^{[q]}}$$
induced by the natural quotient map $\sO_{A_{\mu}} \to \sO_{A_{\mu}}/I^{[q]}$ is surjective. 
Thus, we obtain the assertion. 
\end{proof}

Let $\varphi_{X_{\mu}, e_0}: F^{ne_0}_*\sO_{X_{\mu}}  \to \sO_{X_{\mu}}$ be a generator for the rank-one free $F^{ne_0}_*\sO_{X_{\mu}}$-module $\mathrm{Hom}_{\sO_{X_{\mu}}}(F^{ne_0}_*\sO_{X_{\mu}}, \sO_{X_{\mu}})$. 
Then $\widetilde{\tau}_{D^Y_{\mu}}(X_{\mu}, Z_{\mu} \big|_{X_{\mu}})$ is the unique smallest ideal $J$ whose support does not contain any component of $D^Y_{\mu}$ and which satisfies 
$$\varphi_{X_{\mu}, ne_0}(F^{ne_0}_*(Jg^{(q_0^n-1)/r}\ba_1^{\lceil t_1(q_0^{n}-1) \rceil} \cdots \ba_m^{\lceil t_m(q_0^{n}-1) \rceil})) \subseteq J$$
for all integers $n \ge 1$. 
By Lemma \ref{Fedder} (1), there exist an $\sO_{A_{\mu}}$-linear map $\varphi_{A_{\mu}, ne_0}: F^{ne_0}_*\sO_{A_{\mu}}  \to \sO_{A_{\mu}}$ and a germ $h_n \in \sO_{A_{\mu}}$ whose image is a generator for the cyclic $\sO_{X_{\mu}}$-module $(I^{[q_0^n]}:I)/I^{[q_0]}$ such that we have the following commutative diagram:
$$
\xymatrix{
F^{ne_0}_*\sO_{A_{\mu}} \ar[d] \ar[rrr]^{\varphi_{A_{\mu}, ne_0} \circ F^{ne_0}_*h_n } & & & \sO_{A_{\mu}} \ar[d]\\
F^{ne_0}_*\sO_{X_{\mu}} \ar[rrr]^{\varphi_{X_{\mu}, ne_0}} & & & \sO_{X_{\mu}},
}
$$
where the vertical maps are natural quotient maps. 
By the definition of $\widetilde{\tau}_{X_{\mu}}(A_{\mu},Z_{\mu})$, one has 
$$\varphi_{A_{\mu}, ne_0}(F^{ne_0}_*(\widetilde{\tau}_{X_{\mu}}(A_{\mu},Z_{\mu})I^{c(q_0^n-1)} \ba_1^{\lceil t_1(q_0^{n}-1) \rceil} \cdots \ba_m^{\lceil t_m(q_0^{n}-1) \rceil})) \subseteq \widetilde{\tau}_{X_{\mu}}(A_{\mu},Z_{\mu}).$$
Since $g^{(q_0^{n}-1)/r}  h_n \in I^{c(q_0^n-1)}+I^{[q_0^n]}$ by the above claim, 
$$\varphi_{A_{\mu}, ne_0}(F^{ne_0}_*(\widetilde{\tau}_{X_{\mu}}(A_{\mu},Z_{\mu})g^{(q_0^{n}-1)/r}  h_n \ba_1^{\lceil t_1(q_0^{n}-1) \rceil} \cdots \ba_m^{\lceil t_m(q_0^{n}-1) \rceil})) \subseteq \widetilde{\tau}_{X_{\mu}}(A_{\mu},Z_{\mu})+I.$$
It then follows from the commutativity of the above diagram that 
$$\varphi_{X_{\mu}, ne_0}(F^{ne_0}_*(\widetilde{\tau}_{X_{\mu}}(A_{\mu},Z_{\mu})  \big|_{X_{\mu}} \overline{g}^{(q_0^{n}-1)/r} \overline{\ba_1}^{\lceil t_1(q_0^{n}-1) \rceil} \cdots \overline{\ba_m}^{\lceil t_m(q_0^{n}-1) \rceil})) \subseteq \widetilde{\tau}_{X_{\mu}}(A_{\mu},Z_{\mu}) \big|_{X_{\mu}},$$
where $\overline{\ba_i}$ is the image of $\ba_i$ in $\sO_{X_{\mu}}$ for each $i=1, \dots, m$. 

On the other hand, note that $\ba_1^{\lceil t_1 \rceil} \cdots \ba_m^{\lceil t_m \rceil} \widetilde{\tau}_{X_{\mu}}(A_{\mu}) \subseteq \widetilde{\tau}_{X_{\mu}}(A_{\mu},Z_{\mu})$. 
By \cite[Example 2.6]{Ta2}, the support of $\widetilde{\tau}_{X_{\mu}}(A_{\mu})$ is contained in the singular locus of $X_{\mu}$, which does not contain any component of $D^Y_{\mu}$ because $X_{\mu}$ is normal. 
Also, by a general choice of  $Y$, we may assume that no component of $D^Y_{\mu}$ is contained in the support of $Z_{i, \mu}$ for all $i=1, \dots, m$. 
Thus, the support of $\widetilde{\tau}_{X_{\mu}}(A_{\mu},Z_{\mu}) \big|_{X_{\mu}}$ does not contain any component of $D^Y_{\mu}$.  
By the minimality of $\widetilde{\tau}_{D^Y_{\mu}}(X_{\mu}, Z_{\mu} \big|_{X_{\mu}})$, we conclude that $\widetilde{\tau}_{D^Y_{\mu}}(X_{\mu}, Z_{\mu} \big|_{X_{\mu}}) \subseteq \widetilde{\tau}_{X_{\mu}}(A_{\mu},Z_{\mu}) \big|_{X_{\mu}}$. 
\end{proof}

\begin{rem}\label{lc rem}
Let the notation be as in Theorem \ref{restriction} and fix an arbitrary point $x \in X$. 
Employing the same strategy as the proof of \cite[Theorem]{Ka2}, we can use Theorem \ref{restriction} to prove that the pair $(X; Z\big|_X+\frac{1}{r}  V(J_X))$ is log canonical at $x$ if and only if so is $(A; c X +Z)$.
This result is a special case of \cite[Theorem 1.1]{Ka} and \cite[Theorem 1.1]{EM}, but our proof does not depend on the theory of jet schemes. 
\end{rem}

As a corollary, we prove the conjecture proposed in \cite[Conjecture 2.8]{Ta2} when $X$ is normal and $\Q$-Gorenstein. 
\begin{cor}\label{correspondence} 
Let $A$ be a smooth variety over an algebraically closed field $K$ of characteristic zero 
and $X \subseteq A$ be a normal $\Q$-Gorenstein closed subvariety of $A$. 
Let $Z=\sum_{i=1}^m t_i Z_i$ be a formal combination where the $t_i$ are nonnegative real numbers and the $Z_i \subseteq A$ are proper closed subschemes which do not contain $X$ in their support. 
Given any model of $(A, X, Z)$ over a finitely generated $\Z$-subalgebra $B$ of $K$, 
there exists a dense open subset $W \subseteq \Spec B$ such that  
$$\adj_X(A,Z)_{\mu}=\widetilde{\tau}_{X_{\mu}}(A_{\mu},Z_{\mu})$$
for every closed point $\mu \in W$. 
In particular, the pair $(A; Z)$ is plt along $X$ if and only if it is of purely $F$-regular type along $X$. 
\end{cor}
\begin{proof}
Let $r$ be the Gorenstein index of $X$ and $J_X \subseteq \sO_X$ be the l.c.i.~defect ideal sheaf  of $X$. 
Let $(A_B, X_B, Z_B, J_{X, B})$ be any model of $(A, X, Z, J_X)$ over a finitely generated $\Z$-subalgebra $B$ of $K$. 
By \cite[Theorem 2.7]{Ta2},  there exists a dense open subset $W \subseteq \Spec B$ such that 
$$\widetilde{\tau}_{X_{\mu}}(A_{\mu}, Z_{\mu}) \subseteq \adj_X(A,Z)_{\mu}$$
for all closed points $\mu \in W$.
Therefore, we will prove the reverse inclusion. 
 
As an application of Theorem \ref{mult thm} to $(\star)$ and $(\star\star)$ in the proof of Theorem \ref{restriction}, after replacing $W$ by a smaller dense open subset if necessary, we may assume that 
$$\adj_X(A,Z)_{\mu}\big|_{X_{\mu}}=\J(X, Z \big|_X+\frac{1}{r}  V(J_X))_{\mu} 
\subseteq \widetilde{\tau}_{X_{\mu}}(A_{\mu},Z_{\mu})\big|_{X_{\mu}},
$$
that is, 
$$\adj_X(A,Z)_{\mu} \subseteq \widetilde{\tau}_{X_{\mu}}(A_{\mu},Z_{\mu})+\mathcal{I}_{X_{\mu}}$$
for all closed points $\mu \in W$. 
It, however, follows from Theorem \ref{mult thm} and \cite[Theorem 5.1]{Ei} that we may assume that 
\begin{align*}
\adj_X(A,Z)_{\mu} \cap \mathcal{I}_{X_{\mu}}=\J(A, cX+Z)_{\mu}&=\widetilde{\tau}(A_{\mu}, cX_{\mu}+Z_{\mu}) \\
&\subseteq \widetilde{\tau}_{X_{\mu}}(A_{\mu},Z_{\mu})
\end{align*}
for all closed points $\mu \in W$.
Thus, $\adj_X(A,Z)_{\mu} \subseteq \widetilde{\tau}_{X_{\mu}}(A_{\mu},Z_{\mu})$ for all closed points $\mu \in W$.
\end{proof}


\section{The correspondence of log canonicity and F-purity when the defining equations are very general}
Using the argument developed in the previous section and involving the l.c.i.~defect ideal sheaf,  we will show that Conjecture \ref{lc conj} holds true if the defining equations of the variety are very general. 
The following result is a generalization of a result of Hern\'andez \cite{He} to the singular case.  
\begin{thm}\label{very general}
Let $\mathbb{A}^n_K=\Spec K[x_1, \dots, x_n]$ be the affine $n$-space over an algebraically closed field $K$ of characteristic zero and $X \subseteq \mathbb{A}^n_K$ be a normal $\Q$-Gorenstein closed subvariety of codimension $c$ passing through the origin $0$. 
Let $r$ denote the Gorenstein index of $X$ and $J_X$ denote the l.c.i.~defect ideal of $X$.
Let $\ba \subseteq \sO_X$ be a nonzero ideal and $t > 0$ be a real number. 
Suppose that there exist a system of generators $h_1, \dots, h_l$ for the defining ideal $\mathcal{I}_X$ of $X$  and a system of generators $h_{l+1}, \dots, h_{\nu}$ for $\ba$ with the following property: for each $i=1, \dots, \nu$, we can write 
$$h_i=\sum_{j=1}^{\rho_i}\gamma_{ij}x_1^{\alpha^{(1)}_{ij}} \cdots x_n^{\alpha^{(n)}_{ij}} \in K[x_1, \dots, x_n]  
\; \left((\alpha^{(1)}_{ij}, \dots, \alpha^{(n)}_{ij}) \in \Z_{\ge 0}^n \setminus \{\mathbf{0}\}, \; \gamma_{ij} \in K^* \right),$$
where $\gamma_{i1}, \dots, \gamma_{i\rho_i}$ are algebraically independent over $\Q$. 
Then $(X; t V(\ba)+\frac{1}{r} V(J_X))$ is log canonical at $0$ if and only if it is of dense sharply $F$-pure type at $0$. 
\end{thm}

\begin{rem}
Note by the definition of $J_X$ that $X$ is l.c.i. if and only if $J_X=\sO_X$. 
Thus, if $X=\Spec K[x_1, \dots, x_n]/(h_1, \dots, h_l)$ is a normal complete intersection variety, $\ba \subseteq \sO_X$ is the image of the ideal generated by $h_{l+1}, \cdots, h_{\nu}$, and the $h_i \in K[x_1, \dots, x_n]$ satisfy the same property as that in Theorem \ref{very general},  then Theorem \ref{very general} says that $(X, t V(\ba))$ is log canonical at $0$ if and only if it is of dense sharply $F$-pure type. 
\end{rem}

\begin{proof}
By Remark \ref{F-pure=>lc}, it suffices to show that if $(X; t V(\ba)+\frac{1}{r} V(J_X))$ is log canonical at $0$, then it is of dense sharply F-pure type.

Suppose that $(X; t  V(\ba)+\frac{1}{r}  V(J_X))$ is log canonical at $0$. 
Since the log canonical threshold $\lct_0((X; \frac{1}{r} V(J_X)); V(\ba))$ is a rational number, we may assume that $t$ is a rational number. 
Take a sufficiently general complete intersection closed subscheme $Y=V((f_1, \dots, f_c))$ of codimension $c$ containing $X$, and let $s=c-l+\nu$ and $f_{c+j}=h_{l+j}$ for every $j=1, \dots, s-c$. 
For each $i=1, \dots, s$, we can write
$$f_i=\sum_{j=1}^{m_i}u_{ij}x_1^{a^{(1)}_{ij}} \cdots x_n^{a^{(n)}_{ij}} \in K[x_1, \dots, x_n]  
\quad \left((a^{(1)}_{ij}, \dots, a^{(n)}_{ij}) \in \Z_{\ge 0}^n \setminus \{\mathbf{0}\}, \; u_{ij} \in K^* \right),$$
where $u_{11}, \dots, u_{1m_1}, \dots, u_{s1}, \dots, u_{sm_s}$ are algebraically independent over $\Q$. 
We decompose $Y$ into the scheme-theoretic union of $X$ and a variety $C^Y$, and denote by $D^Y$ the Weil divisor on $X$ obtained by restricting $C^Y$ to $X$. 
Let $g \in K[x_1, \dots, x_n]$ be a polynomial whose image is a local equation of the Cartier divisor $r D^Y$ in a neighborhood of $0$. 
Using the standard decent theory of \cite[Chapter 2]{HH},  we can choose a model 
$$(\mathbb{A}^n_B=\Spec B[x_1, \dots, x_n], X_B, Y_B=V((f_{1, B}, \dots, f_{c, B})), D^Y_B, \ba_B, J_{X, B}, g_B)$$ of $(\mathbb{A}^n_K, X, Y, D^Y, \ba, J_X, g)$ over a finitely generated $\Z$-subalgebra $B$ of $K$ such that 
\renewcommand{\labelenumi}{(\roman{enumi})}
\begin{enumerate}
\item $\Z[u_{11}, \dots, u_{1m_1}, \dots, u_{s1}, \dots, u_{sm_s}, 1/(\prod_{i,j} u_{ij})] \subseteq B$,
\item the image of $g_B$ lies in $J_B$, 
\item $X_{\mu}$ is a normal $\Q$-Gorenstein closed subvariety of codimension $c$ passing the origin $0$ with Gorenstein index $r$,
\item $Y_{\mu}$ is a complete intersection closed subscheme of codimension $c$ containing $X_{\mu}$, 
\item $rD^Y_{\mu}$ is a Cartier divisor on $X_{\mu}$ and $\sO_{X_{\mu}}(rK_{X_{\mu}})=\sO_{X_{\mu}}(-rD^Y_{\mu})\omega_{Y_{\mu}}^{\otimes r}$, 
\item the image of $g_{\mu}$ is a local equation of $r D^Y_{\mu}$ at $0$
\end{enumerate}
for all closed points $\mu \in \Spec B$. 
It is then enough to show that there exists a dense subset of closed points $W \subseteq \Spec B$ such that $\left(X_{\mu}; t V(\ba_{\mu})+\frac{1}{r} V(J_{X, \mu}) \right)$
is sharply $F$-pure at $0$ for all $\mu \in W$. 

Since $(X; t V(\ba)+\frac{1}{r} V(J_X))$ is log canonical at $0$, it follows from \cite[Theorem 1.1]{Ka} and \cite[Theorem 1.1]{EM} (see also Remark \ref{lc rem}) that $(\mathbb{A}^n_K; t  V(\ba)+c X)$ is log canonical at $0$.  
By a general choice of $f_1, \dots, f_c$, it is equivalent to saying that 
$$(\mathbb{A}^n_K; \sum_{i=1}^c \mathrm{div}(f_i)+t  V(f_{c+1}, \dots, f_s))$$ is log canonical at $0$. 
By making use of  the summation formula for multiplier ideals \cite[Theorem 3.2]{Ta}, for any $\epsilon>0$, there exist nonnegative rational numbers $\lambda_{c+1}(\epsilon), \dots, \lambda_s(\epsilon)$ with $\lambda_{c+1}(\epsilon)+\dots+\lambda_s(\epsilon)=t(1-\epsilon)$ such that 
$$(\mathbb{A}^n_K, \sum_{i=1}^c (1-\epsilon)  \mathrm{div}(f_i)+\sum_{j=c+1}^s \lambda_j(\epsilon)  \mathrm{div}(f_j))$$ is klt at $0$. 
Let $\ba_{f_i}$ be the term ideal of $f_i$ (that is, the monomial ideal generated by the terms of $f_i$) for each $i=1, \dots, s$. 
Since $\ba_{f_i}$ contains $f_i$, the monomial ideal $\J(\mathbb{A}^n_K, \sum_{i=1}^c (1-\epsilon) V(\ba_{f_i})+\sum_{j=c+1}^s \lambda_j(\epsilon) V(\ba_{f_j}))$ is trivial. 
Then by a result of Howald \cite[Main Theorem]{Ho1}, the vector $\mathbf{1}$ lies in the interior of 
$$\sum_{i=1}^c  (1-\epsilon) P(\ba_{f_i})+\sum_{j=c+1}^s \lambda_j(\epsilon) P(\ba_{f_j}),$$ where $P(\ba_{f_i})$ is the Newton Polyhedron of $\ba_{f_i}$ for each $i=1, \dots, s$. 
This is equivalent to saying that there exists 
$$\sigma(\epsilon)=(\sigma_{11}{(\epsilon)}, \dots, \sigma_{1m_1}{(\epsilon)}, \dots, \sigma_{s1}{(\epsilon)}, \dots, \sigma_{sm_s}{(\epsilon)}) \in \R_{\ge 0}^{\sum_{i=1}^s m_i}$$ 
such that 
\begin{enumerate}
\item[$(1)$] 
$A  \sigma(\epsilon)^{\mathrm{T}} < \mathbf{1}$,
\item[$(2)$] 
$\sum_{j=1}^{m_i} \sigma_{ij}{(\epsilon)}=1-\epsilon$ for every $i=1, \dots, c$, 
\item[$(3)$] 
$\sum_{j=1}^{m_i} \sigma_{ij}{(\epsilon)}=\lambda_i(\epsilon)$ for every $i=c+1, \dots, s$, 
\end{enumerate}
where $A$ is the $n \times (\sum_{i=1}^{s} m_i)$ matrix 
$$
\left(
\begin{array}{ccccccccc}
a^{(1)}_{11} & \ldots & a^{(1)}_{1m_1} & a^{(1)}_{21} & \ldots & a^{(1)}_{s1} & \ldots & a^{(1)}_{sm_s}\\
\vdots &\ddots & \vdots & \vdots & & \vdots  & \ddots & \vdots\\
a^{(n)}_{11} & \ldots & a^{(n)}_{1m_1} & a^{(n)}_{21} & \ldots  &  a^{(n)}_{s1} & \ldots & a^{(n)}_{sm_s} \\
\end{array}
\right).
$$
Since such $\sigma(\epsilon)$ exists for every $\epsilon>0$,  
by the continuity of real numbers and the convexity of the solution space $\{\mathbf{\tau} \in \R^{\sum_{i=1}^s m_i}_{\ge 0} \big| \widetilde{A} \mathbf{\tau}^{\rm T} \le \mathbf{1}\}$, there exists 
$$\sigma=(\sigma_{11}, \dots, \sigma_{1m_1}, \dots, \sigma_{s1}, \dots, \sigma_{sm_s}) \in \Q_{\ge 0}^{\sum_{i=1}^s m_i}$$ 
such that 
\begin{enumerate}
\item[$(\widetilde{1})$] 
$\widetilde{A} \sigma^{\mathrm{T}} \le \mathbf{1}$,
\item[$(\widetilde{2})$] 
$\sum_{j=1}^{m_i} \sigma_{ij}=1$ for every $i=1, \dots, c$, 
\item[$(\widetilde{3})$] 
$\sum_{i=c+1}^s\sum_{j=1}^{m_i} \sigma_{ij}=t$, 
 \end{enumerate}
where $\widetilde{A}$ is the $(n+s) \times (\sum_{i=1}^{s} m_i)$ matrix 
$$
\left(
\begin{array}{ccccccccccc}
a^{(1)}_{11} & \ldots & a^{(1)}_{1m_1} & a^{(1)}_{21} & \ldots & a^{(1)}_{2m_2} & a^{(1)}_{31} & \ldots & a^{(1)}_{s1} & \ldots & a^{(1)}_{sm_s}\\
\vdots & \ddots & \vdots & \vdots & \ddots & \vdots & \vdots  & & \vdots & \ddots & \vdots\\
a^{(n)}_{11} & \ldots & a^{(n)}_{1m_1} & a^{(n)}_{21} & \ldots & a^{(n)}_{2m_2} &  a^{(n)}_{31} & \ldots & a^{(n)}_{s1} & \ldots & a^{(n)}_{sm_s} \\
1 & \ldots & 1 & 0 & \ldots & 0 & 0 & \ldots & 0 & \ldots & 0 \\
& &  & 1 & \ldots & 1 & 0 & \ldots & 0 & \ldots & 0\\
& &  & & & & 1 & \dots & 0 & \ldots & 0\\
& & & & & & & & \vdots & & \vdots \\
& \bigzerol & & & & & & & 0 & \ldots & 0 \\
& &   &  &  &  &  & & 1 & \ldots & 1\\
\end{array}
\right).
$$
We take the least common multiple $N$ of the denominators of the $\sigma_{ij}$, so that $\sigma_{ij}(p-1)$ is an integer for all $i=1, \dots, s$ and all $j=1, \dots,  m_i$ whenever $p \equiv 1 \textup{ mod } N$. 

Let $p$ be a prime such that $p \equiv 1 \textup{ mod } Nr$ and let $e_1, \dots, e_n$ be nonnegative integers such that 

$$
(p-1)\widetilde{A}  \sigma^{\rm T}
=
\left(
\begin{array}{c}
e_1\\
\vdots \\
e_n\\
\sum_{j=1}^{m_1}\sigma_{1j}(p-1)\\
\vdots\\
\sum_{j=1}^{m_s}\sigma_{sj}(p-1)\\
\end{array}
\right).
$$
Then $e_k \le p-1$ for all $k=1, \dots, n$. 
The coefficient of the monomial $x_1^{e_1} \cdots x_n^{e_n}$ in the expansion of $f_{1}^{\sum_{j=1}^{m_1}\sigma_{1j}(p-1)} \cdots f_s^{\sum_{j=1}^{m_s}\sigma_{sj}(p-1)}$ is
$$ \theta_{\mathbf{\sigma}, p}(\mathbf{u}):=\sum_{\tau_{ij}} \prod_{i=1}^s  \binom{\sum_{j=1}^{m_i}\sigma_{ij}(p-1)}{\tau_{i1}, \dots, \tau_{im_i}} 
u_{i1}^{\tau_{i1}} \cdots u_{im_i}^{\tau_{im_i}} \in \Z[u_{ij}]_{\hspace*{-0.5em}i=1, \dots, s \atop j=1, \dots, m_i} \subseteq B,$$
where the summation runs over all $\tau=(\tau_{11}, \dots, \tau_{1m_1}, \dots, \tau_{s1}, \dots, \tau_{sm_s}) \in \Z_{\geq 0}^{\sum_{i=1}^s m_i}$ such that 

\noindent
$$
\widetilde{A}  \tau^{\rm T}
=
\left(
\begin{array}{c}
e_1\\
\vdots \\
e_n\\
\sum_{j=1}^{m_1}\sigma_{1j}(p-1)\\
\vdots\\
\sum_{j=1}^{m_s}\sigma_{sj}(p-1)\\
\end{array}
\right).
$$
Since $\widetilde{A}  \sigma^{\mathrm{T}} \le \mathbf{1}$, one has $\sum_{j=1}^{m_i}\sigma_{ij}(p-1) \le p-1$ for all $i=1, \dots, s$, 
so the coefficient   
$$\prod_{i=1}^s \binom{\sum_{j=1}^{m_i}\sigma_{ij}(p-1)}{\sigma_{i1}(p-1), \dots, \sigma_{im_i}(p-1)}$$ 
of the monomial $\prod_{i=1}^s u_{i1}^{\sigma_{i1}(p-1)} \cdots u_{im_i}^{\sigma_{im_i}(p-1)}$ in $\theta_{\mathrm{\sigma},p}(\mathbf{u})$ is nonzero in $\F_p$. 
This means that $\theta_{\mathrm{\sigma},p}(\mathbf{u})$ is nonzero in $\F_p[u_{ij}]_{\hspace*{-0.5em}i=1, \dots, s \atop j=1, \dots, m_i} \subseteq B/pB$, because, by assumption, the $u_{ij}$ are algebraically independent over $\F_p$. 
Thus, $D(\theta_{\mathrm{\sigma},p}(\mathbf{u})) \cap \Spec B/pB$ is a dense open subset of $\Spec B/pB$. 

We now set
$$W:=\bigcup_{p \equiv 1 \textup{ mod } Nr} D(\theta_{\mathrm{\sigma},p}(\mathbf{u})) \cap \Spec B/pB \subseteq \Spec B.$$
Then $W$ is a dense subset of $\Spec B$. 
Fix any closed point $\mu \in W$ and let $p$ denote the characteristic of the residue field $\kappa(\mu)=B/\mu$ from now on. 
Since the image of $\theta_{\mathrm{\sigma},p}(\mathbf{u})$ is nonzero in $B/\mu$, the monomial $x_1^{e_1} \cdots x_n^{e_n}$ appears in the expansion of $f_{1, \mu}^{(\sum_{j=1}^{m_1}\sigma_{1j})(p-1)} \cdots f_{s, \mu}^{(\sum_{j=1}^{m_s}\sigma_{sj})(p-1)}$ in $(B/\mu)[x_1, \dots, x_n]$. 
Since $e_k \leq p-1$ for all $k=1, \dots, n$ and $\sum_{j=1}^{m_i}\sigma_{ij}(p-1)=p-1$ for all $i=1, \dots, c$, one has 
$$f_{1, \mu}^{p-1} \cdots f_{c, \mu}^{p-1}f_{c+1, \mu}^{(\sum_{j=1}^{m_{c+1}}\sigma_{c+1j})(p-1)} \cdots f_{s, \mu}^{(\sum_{j=1}^{m_s}\sigma_{sj})(p-1)} \notin (x_1^p, \dots, x_n^p)$$
in $(B/\mu)[x_1, \dots, x_n]_{(x_1, \dots, x_n)}$. 
By Lemma \ref{Fedder} (2), this is equivalent to saying that for all powers $q=p^e$ of $p$, 
$$f_{1, \mu}^{q-1} \cdots f_{c, \mu}^{q-1}f_{c+1, \mu}^{(\sum_{j=1}^{m_{c+1}}\sigma_{c+1j})(q-1)} \cdots f_{s, \mu}^{(\sum_{j=1}^{m_s}\sigma_{sj})(q-1)} \notin (x_1^q, \dots, x_n^q)$$
in $(B/\mu)[x_1, \dots, x_n]_{(x_1, \dots, x_n)}$. 
Applying the claim in the proof of Theorem \ref{restriction}, one has 
$$(\mathcal{I}_{X, \mu}^{[q]}:\mathcal{I}_{X, \mu})g_{\mu}^{(q-1)/r}f_{c+1, \mu}^{(\sum_{j=1}^{m_{c+1}}\sigma_{c+1j})(q-1)} \cdots f_{s, \mu}^{(\sum_{j=1}^{m_s}\sigma_{sj})(q-1)} \notin (x_1^q, \dots, x_n^q)$$
in $(B/\mu)[x_1, \dots, x_n]_{(x_1, \dots, x_n)}$. 
Since $\sum_{i=c+1}^s\sum_{j=1}^{m_i}\sigma_{ij}=t$ and the image of $g_{\mu}$ lies in $J_{X, \mu}$, it follows from Lemma \ref{Fedder} (2) again that the pair $\left(X_{\mu}; \frac{1}{r} V(J_{X, \mu})+ t  V(\ba_{\mu}) \right)$ is sharply $F$-pure at $0$. 
\end{proof}

\begin{rem}\label{computation}
Using the same arguments as the proof of Theorem \ref{very general}, we can prove the following: 
let $X=\Spec K[x_1, \dots, x_n]/(f_1, \dots, f_c)$ be a normal complete intersection over a field $K$ of characteristic zero passing through the origin $0$. 
Let $Z \subset X$ be a proper closed subscheme passing through $0$ and $f_{c+1}, \dots, f_s$ be a system of polynomials whose image generates the defining ideal $\mathcal{I}_Z \subseteq \sO_X$ of $Z$. 
We write  
$$f_i=\sum_{j=1}^{m_i}u_{ij}x_1^{a^{(1)}_{ij}} \cdots x_n^{a^{(n)}_{ij}} \in K[x_1, \dots, x_n]  
\quad \left((a^{(1)}_{ij}, \dots, a^{(n)}_{ij}) \in \Z_{\ge 0}^n \setminus \{\mathbf{0}\}, \; u_{ij} \in K^* \right)$$
for each $i=1, \dots, s$, and set $A$ to be the $(n+s) \times (\sum_{i=1}^{s} m_i)$ matrix  
$$
\left(
\begin{array}{ccccccccccc}
a^{(1)}_{11} & \ldots & a^{(1)}_{1m_1} & a^{(1)}_{21} & \ldots & a^{(1)}_{2m_2} & a^{(1)}_{31} & \ldots & a^{(1)}_{s1} & \ldots & a^{(1)}_{sm_s}\\
\vdots & \ddots & \vdots & \vdots & \ddots & \vdots & \vdots  & & \vdots & \ddots & \vdots\\
a^{(n)}_{11} & \ldots & a^{(n)}_{1m_1} & a^{(n)}_{21} & \ldots & a^{(n)}_{2m_2} &  a^{(n)}_{31} & \ldots & a^{(n)}_{s1} & \ldots & a^{(n)}_{sm_s} \\
1 & \ldots & 1 & 0 & \ldots & 0 & 0 & \ldots & 0 & \ldots & 0 \\
& &  & 1 & \ldots & 1 & 0 & \ldots & 0 & \ldots & 0\\
& &  & & & & 1 & \dots & 0 & \ldots & 0\\
& & & & & & & & \vdots & & \vdots \\
& \bigzerol & & & & & & & 0 & \ldots & 0 \\
& &   &  &  &  &  & & 1 & \ldots & 1\\
\end{array}
\right).
$$
Then we consider the following linear programming problem:
\begin{align*}
&\textup{Maximize: } \sum_{i={c+1}}^s\sum_{j=1}^{m_i} \sigma_{ij} \\
&\textup{Subject to: } 
A  (\sigma_{11}, \dots, \sigma_{1m_1}, \dots, \sigma_{s1}, \dots, \sigma_{sm_s})^{\mathrm{T}} \le \mathbf{1},\\
&\hspace*{5.1em}\sum_{i=1}^c\sum_{j=1}^{m_i} \sigma_{ij}=c,\\
&\hspace*{5.1em}\sigma_{ij} \in \Q_{\ge 0} \textup{ for all $i=1, \dots, s$ and all $j=1, \dots, m_i$}.
\end{align*}
Assume that there exists an optimal solution $\mathbf{\sigma}=(\sigma_{11}, \dots, \sigma_{1m_1}, \dots, \sigma_{s1}, \dots, \sigma_{sm_s})$ such that $A  \mathbf{\sigma}^{\mathrm{T}} \neq A  {\mathbf{\sigma}'}^{\mathrm{T}}$ for all other optimal solutions ${\mathbf{\sigma}'} \neq \mathbf{\sigma}$.
In addition, we assume that $X$ is log canonical at $0$. 
Then the following hold:
\begin{enumerate}
\item $\lct_0(X, Z)$ is equal to the optimal value $\sum_{i=c+1}^s\sum_{j=1}^{m_i} \sigma_{ij}$. 
\item Given any model of $(X, Z)$ over a finitely generated $\Z$-subalgebra $B$ of $K$, 
there exists a dense subset of closed points $W \subseteq \Spec B$ such that 
$$\mathrm{lct}_0\left(X; Z \right)=\mathrm{fpt}_{0}\left(X_{\mu}; Z_{\mu} \right)$$
for all $\mu \in W$. 
\end{enumerate}
\end{rem}

In \cite{ST}, Shibuta and the author showed that the assumption of Remark \ref{computation} is satisfied if $X=\mathbb{A}^n_K$ and $Z$ is a complete intersection binomial subscheme or a space monomial curve (in the latter case, $n=3$). However, in general, there exists a binomial subscheme that does not satisfy the assumption.

\begin{eg}
Let $X=\mathbb{A}^6_K=\Spec K[x_1, x_2, x_3, y_1, y_2, y_3]$ be the affine $6$-space over a field $K$ of characteristic zero and $Z \subseteq X$ be the closed subscheme defined by the binomials $x_1y_2-x_2y_1, x_2y_3-x_3y_2$ and $x_1y_3-x_3y_1$.  
Then $Z$ does not satisfy the assumption of Remark \ref{computation}. 
Indeed, $\lct_0(X, Z)=2$ but the optimal value of the linear programming problem in Remark \ref{computation} is equal to $3$.  
Given a prime number $p$, let $X_p=\mathbb{A}^6_{\F_p}=\Spec \F_p[x_1, x_2, x_3, y_1, y_2, y_3]$ and $Z_p \subseteq X_p$ be the reduction modulo $p$ of $Z$.
Since $\fpt_0(X_p, Z_p)=2$ for all primes $p$, Conjecture \ref{lc conj} holds for this example.  
\end{eg}

\end{document}